\documentclass[10pt]{article}

\usepackage{amsmath}
\usepackage{amsfonts}
\usepackage{amsthm}
\usepackage{amssymb}
\usepackage[english]{babel}
\usepackage{graphicx}
\usepackage[all]{xy}

\newtheorem{theorem}{Theorem}
\newtheorem{lemma}{Lemma}
\newtheorem{definition}{Definition}
\newtheorem{corollary}{Corollary}

\newtheorem{remark}{Remark}

\newcommand{\topa}{\circledcirc_1}
\newcommand{\topb}{\circledcirc_2}
\newcommand{\Hom}{{\mathrm{Hom}}}

\newcommand{\mR}{\mathbb{R}}
\newcommand{\mC}{\mathbb{C}}
\newcommand{\mN}{\mathbb{N}}
\newcommand{\mE}{\mathbb{E}}
\newcommand{\mZ}{\mathbb{Z}}
\newcommand{\mS}{\mathbb{S}}

\newcommand{\mD}{\mathbb{D}}

\newcommand{\mg}{\mathfrak{g}}
\newcommand{\mn}{\mathfrak{n}}

\newcommand{\cD}{\mathcal{D}}

\newcommand{\cH}{\mathcal{H}}

\newcommand{\cF}{\mathcal{F}}
\newcommand{\cP}{\mathcal{P}}

\newcommand{\cK}{\mathcal{K}}

\newcommand{\cV}{\mathcal{V}}

\newcommand{\cU}{\mathcal{U}}

\newcommand{\cA}{\mathcal{A}}
\newcommand{\cJ}{\mathcal{J}}

\newcommand{\osp}{\mathfrak{osp}(m|2n)}

\begin{document}
%%%%%%%%%%%%%%%%%%%%%%%%%%%%%%%%%%%%%%%%%%%%%%%%%%%%%%%%
\title{Joseph-like ideals and harmonic analysis for $\osp$}

\author{K.\ Coulembier\thanks{Ph.D. Fellow of the Research Foundation - Flanders (FWO), E-mail: {\tt Coulembier@cage.ugent.be}}, P. Somberg\thanks{ E-mail: {\tt soucek@karlin.mff.cuni.cz}}  and V. Soucek\thanks{ E-mail: {\tt somberg@karlin.mff.cuni.cz}}}

\date{\small{Department of Mathematical Analysis}\\
\small{Department of Mathematics -- Ghent University\\ Krijgslaan 281, 9000 Gent,
Belgium}\\
}

\maketitle

\begin{abstract}
The Joseph ideal in the universal enveloping algebra ${\fam2 U}({\mathfrak so}(m))$ is the annihilator ideal of the $\mathfrak{so}(m)$-representation on the harmonic functions on $\mR^{m-2}$. 
The Joseph ideal for $\mathfrak{sp}(2n)$ is the annihilator ideal of the Segal-Shale-Weil (metaplectic) representation. Both ideals can be constructed in a unified way from a quadratic relation in the 
tensor algebra $\otimes\mg$ for $\mg$ equal to $\mathfrak{so}(m)$ or $\mathfrak{sp}(2n)$. In this paper we construct 
two analogous ideals in $\otimes\mg$ and $\cU(\mg)$ for $\mg$ the orthosymplectic Lie superalgebra $\osp =\mathfrak{spo}(2n|m)$ and prove that they have unique characterizations that naturally extend the classical case. 
Then we show that these two ideals are the annihilator ideals of respectively the 
$\mathfrak{osp}(m|2n)$-representation on the spherical harmonics on $\mR^{m-2|2n}$ and a generalization of the 
metaplectic representation to $\mathfrak{spo}(2n|m)$. This proves that these ideals are 
reasonable candidates to establish the theory of Joseph-like ideals for Lie superalgebras. We also discuss the relation 
between the Joseph ideal of $\osp$ and the algebra of symmetries of the super conformal Laplace operator, regarded as an intertwining operator between principal series representations for $\osp$.
\end{abstract}

\textbf{MSC 2010 :}   17B35, 16S32, 58C50
\noindent
\textbf{Keywords :} Joseph ideal, orthosymplectic superalgebra, superharmonic functions, 
adjoint representation, tensor product, Cartan product
\tableofcontents

\section{Introduction}

Harmonic analysis on manifolds allows to connect and relate various topics in 
geometrical ana-lysis, algebra and representation theory. The present article attempts
to apply this principle in describing and characterizing various algebraic structures
emerging in harmonic analysis of an intertwinning operator for orthosymplectic Lie 
superalgebras, called the super Laplace operator. 
In particular, we study the annihilator ideal of the representation on its kernel on a big cell in the supersphere $\mS^{m-2|2n}$, regarded as a super homogeneous space for $\osp$. We also 
consider the question of an algebraic structure on the space of differential operators 
preserving the representation in question. This is done by constructing a 
Joseph-like ideal for $\mathfrak{osp}(m|2n)$, generalizing the Joseph ideal of $\mathfrak{so}(m)$,
appearing in conformal geometry. In doing so, we also obtain a second 
Joseph-like ideal in this Lie superalgebra, which is a generalization of the Joseph ideal
for $\mathfrak{sp}(2n)$.

The Joseph ideal for a simple complex {\em Lie algebra} $\mg$, not of type A, is a completely prime primitive two-sided ideal in the universal enveloping algebra $\cU(\mg)$. It can be characterized in several ways, see \cite{MR1631302, MR2106228, Garfinkle, MR0342049, MR0953165}. In particular it is the unique completely prime two-sided ideal such that the associated variety is the closure of the minimal nilpotent coadjoint orbit. It is also the annihilator ideal of the unitarizable representation of $\mg$ which has minimal Gelfand-Kirillov dimension, known as the minimal representation. 

The two prominent examples are the Joseph ideals for $\mathfrak{so}(m)$ and $\mathfrak{sp}(2n)$. For the orthogonal case the corresponding minimal representation is the one realized by the conformal algebra acting on harmonic functions on $\mS^{m-2}$, or its flat submanifold $\mR^{m-2}$, see e.g. \cite{MR1108044, MR0342049, MR2020550} for explicit descriptions in several signatures. This ideal subsequently plays an essential role in the description of the symmetries of the Laplace operator, for which its kernel is exactly this minimal representation, see \cite{MR2180410}. The minimal representation for $\mathfrak{sp}(2n)$ is known as the metaplectic representation, Segal-Shale-Weil representation or the symplectic spinors, see \cite{Hilgert2, MR0400304, MR0137504}. This is a representation of $\mathfrak{sp}(2n)$ on functions on $\mR^n$.

In \cite{MR2369839}, based on abstract considerations in \cite{MR1631302}, a $1$-parameter family of nonhomogeneous quadratic ideals in ${\fam2 U}({\mathfrak g})$ for classical Lie algebras ${\mathfrak g}$ equal to $\mathfrak{so}(m)$, $\mathfrak{sp}(2n)$ and $\mathfrak{sl}(n)$ was studied in 
the framework of deformation theory. A certain special tensor 
emerging in the process established finite codimension 
of the ideal except for one special value of the parameter. 
For this special value, the ideal is identical to the Joseph ideal for $\mathfrak{so}(m)$ and $\mathfrak{sp}(2n)$, which gave another unique characterization of the Joseph ideal. 
One of the motivations for the work in \cite{MR2369839} was the Poincar\'e-Birkhoff-Witt theorem in \cite{MR1631302, som} for quadratic Koszul 
algebras. There, the special tensor responsible for deformation theory of this class 
of algebras, was identified as an obstruction class in a Hochschild cohomology.

For a classical simple {\em Lie superalgebra} ${\mathfrak g}$, the universal enveloping algebra 
${\fam2 U}({\mathfrak g})$ never contains a completely prime primitive ideal, different from the augmentation
ideal, unless ${\mathfrak g}$ is isomorphic to the orthosymplectic Lie superalgebra 
$\mathfrak{osp}(1|2n)$, see \cite{MR1616633}. In spite of the geometric origin of $\mathfrak{osp}(m|2n)$
as a conformal Lie superalgebra on $\mR^{m-2|2n}$, the previous argument implies that there is no direct counterpart of the Joseph ideal of classical Lie algebras.
So one has to relax the demanded properties. In particular it follows that not every characterization of the classical case can be extended to the category of superalgebras. As we will show, the approach of \cite{MR2369839} does extend very naturally to the Lie superalgebra $\mathfrak{osp}(m|2n)$. As a side result we also prove that the obtained ideals are also uniquely characterized by an analogue of the elegant characterization of the Joseph ideal in \cite{Garfinkle}.

The Lie superalgebra $\osp$ has two preferred positive root systems, one is logically identified with the $\osp$-interpretation and the other with the $\mathfrak{spo}(2n|m)$-interpretation of this Lie superalgebra, from now on denoted by $\mg$. Corresponding to these two choices a different application of the procedure in \cite{MR2369839} (or the characterization in \cite{Garfinkle}) emerges, which leads to two different ideals in the universal enveloping algebra $\cU(\mg)$. These two ideals can be identified with generalizations of the Joseph ideals of $\mathfrak{so}(m)$ and $\mathfrak{sp}(2n)$. 

In \cite{OSpHarm, MR2344451} the super Laplace operator on $\mR^{p|2n}$ and its kernel were studied. In the current paper we prove that for $m=p+2$, this kernel constitutes an $\osp$-representation, for which the annihilator ideal in the universal enveloping algebra is exactly the first Joseph-like ideal we constructed, if $m-2n>2$. In \cite{Tensor, MR1132090} a spinor-representation for $\mathfrak{spo}(2n|m)$ was studied which generalizes the metaplectic representation of $\mathfrak{sp}(2n)$ and the spinor representation of $\mathfrak{so}(m)$. In particular it was proven in \cite{Tensor} that these spinor spaces are the unique completely pointed $\mathfrak{spo}(2n|m)$-modules. We will also prove that the corresponding annihilator ideal is exactly the second Joseph-like ideal.

Because these two ideals and their corresponding representations possess these remarkable similarities with the Joseph ideals and minimal representations of $\mathfrak{so}(m)$ and $\mathfrak{sp}(2n)$, we believe this
is a good starting point to study the general concept of Joseph ideals and minimal representations for Lie superalgebras. 
In particular it would also be of interest to construct general theories as in 
\cite{MR1108044, MR1631302, Hilgert2, MR0953165} for Lie superalgebras. For instance there seems to exist an interesting link with (co)adjoint orbits, see \cite{MR2264065}. As will become apparent in the current paper, the representations corresponding to the Joseph ideal for Lie superalgebras can not always be expected to be unitarizable. This implies that also there, other requirements compared to the classical case will need to be considered. We hope that our results help to formulate the appropriate properties of Joseph ideals.
   
The structure of the article looks as follows. In Section 2 we recall all 
conventions used in the subsequent sections. Being geometrically motivated, our 
exposition follows the tensor or abstract index notation.

In Section 3 we focus on the description of the second tensor power of the adjoint representation of the Lie superalgebra $\mathfrak{osp}(m|2n)$. Except for $M=m-2n$ equal to $0$, $1$ or $2$, this tensor product is completely reducible.
We identify the irreducible components as highest weight representations for both positive root systems, and construct projectors on these irreducibles. The two root systems we use, are the ones that only contain one isotropic root in their Dynkin diagram. There is a corresponding choice of which submodule in tensor products has to be identified as the Cartan product. If the ordering is introduced in which the roots of $\mathfrak{so}(m)$ are greater 
than those of $\mathfrak{sp}(2n)$, as in \cite{OSpHarm, Tensor}, the Cartan product resembles the one inside $\mathfrak{so}(m)\otimes\mathfrak{so}(m)$. If the standard choice of positive roots is made, as in \cite{MR1773773, MR051963}, the Cartan product resembles the one in $\mathfrak{sp}(2n)\otimes\mathfrak{sp}(2n)$. Therefore, when the standard root system is considered we will use the notation $\mathfrak{spo}(2n|m)$ rather than $\mathfrak{osp}(m|2n)$.

In Section \ref{secCartan} we investigate the Cartan product of the tensor powers of irreducible representations of a semisimple Lie (super)algebra. We propose a general statement which we prove for the natural representation of $\osp$, the adjoint representation of $\osp$, all star-representations of semisimple Lie superalgebras (see \cite{MR0424886}) and all representations of complex semisimple Lie algebras. As a side result we obtain the proof of a conjecture made in \cite{MR2130630}.

In Sections \ref{ospsec} and \ref{sposec} we define and study the two Joseph-like ideals. Though in the current article (contrary to \cite{MR1631302, som}) we shall 
not study Poincar\'e-Birkhoff-Witt theorems for quadratic superalgebras 
of Koszul type, we construct two elements in the tensor algebra of the orthosymplectic
Lie superalgebra which resemble the ones in \cite{MR2369839}. They are given in Lemma \ref{dimhom1} and Lemma \ref{dimhom2}. The reason and necessity for
the appearance of two, rather than one, such elements comes from the fact that the notion of the Cartan product of $\mg\otimes\mg$ depends on the used positive root system. Based on these tensors, we compute the special values in two $1$-parameter families of 
nonhomogeneous quadratic ideals in $\otimes\osp$ or $\cU(\osp)$, for which the ideals are of infinite codimension, in Theorem \ref{1cases} and Theorem \ref{cases2}. These two obtained ideals will be referred to as the Joseph ideal of $\osp$ and $\mathfrak{spo}(2n|m)$ respectively. Furthermore we extend the unique characterization of the classical Joseph ideal, obtained in \cite{Garfinkle}, to these two ideals in Theorem \ref{Garresult} and Theorem \ref{Garresult2}.

The first ideal is consequently identified with the annihilator ideal for the representation on the space of superharmonic functions, which is the content of Subsection \ref{secreposp}. This generalizes the appearance of the Joseph ideal of $\mathfrak{so}(m)$
as the annihilator ideal of the conformal representation on harmonic functions, see e.g. \cite{MR1108044, MR2180410, MR0342049, MR2020550}. In Subsection \ref{secrepspo}, the super metaplectic representation of \cite{Tensor, MR1132090} is considered. We show that the annihilator ideal of this representation is given by the second type of Joseph-like ideal we constructed. The main results are stated in Theorem \ref{idealrep} and Theorem \ref{JosephKostant}.

 Inspired by classical results on the symmetries of the Laplace operator, see \cite{MR2180410, Higherpowers}, we touch 
analogous questions in the last Section \ref{SymmLapl}. It follows that the quotient of the universal enveloping algebra of $\osp$ with respect to the Joseph-like ideal forms an algebra of symmetries. This also leads to a connection between symmetries and superconformal Killing tensor fields. However, orthosymplectic Lie superalgebras suffer from the lack of existence 
of generalized Bernstein-Gelfand-Gelfand resolutions for finite dimensional modules. This implies that completeness of the constructed symmetries does not follow so we can not draw sharper conclusions at this stage. 
 
It is perhaps worth to remark that the present article should be regarded not only as 
a result in Lie theory and representation theory of Lie superalgebras, but also as a germ 
of general strategy to extend the framework of parabolic geometry (\cite{cs}) or ambient metric 
construction (\cite{fg}) to the realm of curved differentiable supermanifolds with a geometrical
(e.g., a conformal) structure. We hope to develop the present results further in this direction.

%%%%%%%%%%%%%%%%%%%%%%%%%%%%%%%%%%%%%%%%%%%%%%%%%%%%%%%%%%%%%%%%%%%%%%%%%%%%%%%%%%%%%%%%%%

\section{Preliminaries and Conventions}
\label{preli}
In this section we introduce a basic collection of conventions with emphasis on the tensor notation approach.

All throughout the paper, the natural numbers are assumed to include $0$, $\mN=\{0,1,2,\cdots\}$. 

The standard basis of the graded complex vector space $V=\mC^{m|2n}$ consists of the vectors $e_a$ for $1\le a\le m+2n$, where $e_a=(0,\cdots,0,1,0,\cdots,0)$ with $1$ at the $a$-th position. The elements $e_a$ with $1\le a\le m$ span $V_0$, and $e_a$ 
with $m<a\le m+2n$ span $V_1$. 
For any $\mZ_2$-graded vector space $V=V_0\oplus V_1$, a vector $u\in V_0\cup V_1$ is called homogeneous and in this case we define 
$|u|=\alpha$ for $u\in V_\alpha$, $\alpha\in\mZ_2$. 
The function
\begin{eqnarray*}
[\cdot]:\{1,2,\cdots,m+2n\}\to\mZ_2, \quad
\text{$[a]=\overline{0}$ if $a\le m$ and $[a]=\overline{1}$ otherwise,}
\end{eqnarray*}
allows to write $|e_a|=[a]$ for all $a$ if $V=\mC^{m|2n}$. In this paper, the summation $\sum_a$ will always stand for $\sum_{a=1}^{m+2n}$.

The endomorphism ring of $\mC^{m|2n}$ is $End(\mC^{m|2n})$, as an associative algebra this is isomorphic to $End(\mC^{m+2n})$, it is a superalgebra with grading induced from $\mC^{m|2n}$. The graded Lie bracket on $End(\mC^{m+2n})$ is given by $[A,B]=A\circ B-(-1)^{|A||B|}B\circ A$ for homogeneous elements $A,B$ and extended by linearity, this makes the vector space $End(\mC^{m|2n})$ into a Lie superalgebra and then it is denoted by $\mathfrak{gl}(m|2n)=\mathfrak{gl}(m|2n;\mC)$. To an element $A\in\mathfrak{gl}(m|2n;\mC)$ we associate a tensor in $\mC^{m|2n}\otimes(\mC^{m|2n})^\ast$ by 
\begin{eqnarray}
\label{tensoraction}
& & Ae_b=\sum_{a}A^{a}{}_be_a,\qquad\mbox{or}
\nonumber \\
& & v^a\mapsto \sum_b A^{a}{}_bv^b,\, \mbox{ for }v=\sum_av^ae_a\in\mC^{m|2n}. 
\end{eqnarray}
To define the subalgebra $\osp =\mathfrak{osp}(m|2n;\mC)\subset\mathfrak{gl}(m|2n;\mC)$, we introduce an orthosymplectic metric $g=g_{ab}\in\mC^{(m+2n)\times(m+2n)}$. This metric is even, i.e. $g_{ab}=0$ if $[a]+[b]=1$, and super symmetric, $g_{ba}=(-1)^{[a][b]}g_{ab}=(-1)^{[a]}g_{ab}$. The metric allows to raise and lower tensor indices by $V_a=\sum_bg_{ab}V^b$ and $V^c=\sum_{a}g^{ac}V_a$ with $g^{ab}=g_{ab}$. 
When considering the orthosymplectic Lie superalgebra $\osp$ we will always assume $m>4$ and $n>1$, as in the classical case for $\mathfrak{so}(m)$ and $\mathfrak{sp}(2n)$, see \cite{MR1108044, MR2369839, MR0342049}.

The orthosymplectic Lie superalgebra can be defined in several ways as the algebra preserving an inner product or element in the tensor space, we restrict to the subsequent definition.
\begin{definition}
The Lie superalgebra $\mathfrak{osp}(m|2n)$ is given by endomorphisms $A\in End(\mC^{m|2n})$ satisfying
\begin{eqnarray*}
\sum_{c}\left(g_{ac}A^{c}{}_b+(-1)^{[b]([a]+[c])}A^{c}{}_ag_{cb}\right)=0.
\end{eqnarray*}
This is equivalent to $A_{ab}=-(-1)^{[a][b]}A_{ba}$ or $A^{ab}=-(-1)^{[a][b]}A^{ba}$.
\end{definition}
This means that, as a graded vector space, $\osp$ is equivalent to the super anti-symmetric tensors in $\mC^{m|2n}\otimes \mC^{m|2n}$ (or equivalently in $\left(\mC^{m|2n}\right)^\ast\otimes \left(\mC^{m|2n}\right)^\ast$).
The trace of a tensor is defined as $\sum_aT_a{}^a=\sum_{a,b}g_{ab}T^{ba}$ and one can calculate
\begin{eqnarray}
\label{traces}
\sum_ag_a{}^a=m-2n=M &\quad\mbox{and}\quad &\sum_{a} A_a{}^a=0 \mbox{ for }A\in\osp .
\end{eqnarray}
The integer $M=m-2n$ will play an important role in the main results of the article.
The Lie superbracket of $A,B\in\mathfrak{osp}(m|2n)$, $C=[A,B]$, can be written in 
tensorial notation as
\begin{eqnarray}
\label{Liebracket}
C^{ab}&=&\sum_{c}\left(A^{a}{}_cB^{cb}-(-1)^{[a][b]}A^b{}_cB^{ca}\right)=\sum_{c}\left(A^{a}{}_{c}B^{cb}+(-1)^{[a]([b]+[c])}A^{b}{}_cB^{ac}\right)
\end{eqnarray}

We use a renormalization of the Killing form, which is not zero if $M=2$, (the case $D(n+1|n)=\mathfrak{osp}(2n+2|2n)$), contrary to the actual Killing form.
\begin{lemma}
\label{Killing}
The Killing bilinear form on $\osp$ in the tensorial notation is proportional to the form
\begin{eqnarray*}
\langle U,V\rangle&=&\sum_{a,b}U_{ab}V^{ba},
\end{eqnarray*}
for $U,V$ $\in\osp$.
\end{lemma}
\begin{proof}
The proposed bilinear form is invariant, $\langle [U,W],V\rangle=\langle U,[W,V]\rangle$, as follows from \eqref{Liebracket} and therefore it is proportional to the Killing form, see Chapter 23 in \cite{MR1773773}.
\end{proof}

As in \cite{OSpHarm, Tensor} we will use two different root systems. These are the two systems that contain only one isotropic simple root. We need to make a distinction between $m=2d$ and $m=2d+1$. All roots for $D(d|n)=\mathfrak{osp}(2d|2n)$ are given by $\pm\epsilon_j\pm\epsilon_k$ for $1\le j<k\le d$, $\pm\delta_i\pm\delta_l$ for $1\le i\le l\le n$ and $\pm\epsilon_j\pm\delta_i$ for $1\le j\le d$ and $1\le i\le n$. The Lie superalgebra $B(d|n)=\mathfrak{osp}(2d+1|2n)$ has in addition the roots $\pm\epsilon_j$ and $\pm \delta_i$. The $\delta$'s correspond to the roots for $\mathfrak{sp}(2n)$ and the $\epsilon$'s to the roots for $\mathfrak{so}(m)$.

For the distinguished standard root system, see \cite{MR1773773, MR051963}, the simple positive roots are given by
\begin{eqnarray}
\label{simpleroots2}
\delta_1-\delta_2,\cdots,\delta_{n-1}-\delta_n,\delta_n-\epsilon_1,\epsilon_1-\epsilon_2,\cdots,\epsilon_{d-1}-\epsilon_d,\epsilon_d,\\
\nonumber
\delta_1-\delta_2,\cdots,\delta_{n-1}-\delta_n,\delta_n-\epsilon_1,\epsilon_1-\epsilon_2,\cdots,\epsilon_{d-1}-\epsilon_d,\epsilon_{d-1}+\epsilon_d
\end{eqnarray}
for $B(d|n)=\mathfrak{osp}(2d+1|2n)$ and $D(d|n)=\mathfrak{osp}(2d|2n)$ respectively.

For the non-standard one, see \cite{OSpHarm, Tensor}, the simple positive roots are given by
\begin{eqnarray}
\label{simpleroots}
\epsilon_1-\epsilon_2,\cdots,\epsilon_{d-1}-\epsilon_d,\epsilon_d-\delta_1,\delta_1-\delta_2,\cdots,\delta_{n-1}-\delta_n,\delta_n\\
\nonumber
\epsilon_1-\epsilon_2,\cdots,\epsilon_{d-1}-\epsilon_d,\epsilon_d-\delta_1,\delta_1-\delta_2,\cdots,\delta_{n-1}-\delta_n,2\delta_n
\end{eqnarray}
for $B(d|n)=\mathfrak{osp}(2d+1|2n)$ and $D(d|n)=\mathfrak{osp}(2d|2n)$ respectively. When it is not mentioned explicitly we use the second root system.

An irreducible highest weight representation for $\osp$ will be denoted by $L_{\lambda}^{m|2n}$ if it has highest weight $\lambda$ with respect to the standard root system. The same representation has a different highest weight $\mu$ with respect to the non-standard root system, and 
will be denoted by $K_\mu^{m|2n}$. The highest weights $\mu$ respectively $\lambda$ can be calculated elegantly from each other through the technique 
of odd reflections, see \cite{MR1327543}. The link between the two weights for this specific case is given explicitly in Theorem 3 of \cite{Tensor} for all irreducible finite dimensional highest weight representations.

We also could have defined the orthosymplectic Lie superalgebra by a natural action on $\mC^{2n|m}$, rather than through equation \eqref{tensoraction}, therefore with $\mathfrak{sp}(2n)$ acting on the even part and $\mathfrak{so}(m)$ acting on the odd part of the super vector space. This leads to the Lie superalgebra $\mathfrak{spo}(2n|m)$, corresponding to super symmetric tensors in $\mC^{2n|m}\otimes\mC^{2n|m}$. However, as a Lie superalgebra $\mathfrak{spo}(2n|m)$ is isomorphic to $\mathfrak{osp}(m|2n)$. We will use the notation $\mathfrak{spo}(2n|m)$ when we use the standard root system and $\mathfrak{osp}(m|2n)$ for the non-standard one, since it will become apparent that this is the logical association of root systems and fundamental representations, e.g. in Remark \ref{relativeCartan}. The fundamental representations then become
\begin{eqnarray*}
\mC^{m|2n}\,\cong\, K_{\epsilon_1}^{m|2n}&=&L_{\delta_1}^{m|2n}\,\cong\,\mC^{2n|m}.
\end{eqnarray*}

For a Lie superalgebra $\mathfrak{g}=\mathfrak{g}_{\overline{0}}\oplus\mg_{\overline{1}}$, the tensor product of two $\mathfrak{g}$-modules $\cU$ and $\cV$, $\cU\otimes \cV$ is a $\mg$-module with action given by
\begin{eqnarray*}
X\cdot(U\otimes V)&=& (X\cdot U)\otimes V+(-1)^{|X||U|}U\otimes (X\cdot V),
\end{eqnarray*}
for $X\in\mg_i$, $U\in \cU_j$ with $i,j\in\mZ_2$ and $V\in\cV$. It is a general fact that the tensor power $\cU\otimes\cU$ 
of a $\mg$-representation $\cU$ for a Lie superalgebra $\mg$
decomposes into the super symmetric part 
\begin{eqnarray*}
\cU\odot\cU=\mbox{span}\{X\otimes Y +(-1)^{|X||Y|}Y\otimes X|\,X,Y\in\cU_{i}\mbox{ for }i=\overline{0},\overline{1}\}
\end{eqnarray*}
and super anti-symmetric part 
\begin{eqnarray*}
\cU\wedge\cU=\mbox{span}\{X\otimes Y -(-1)^{|X||Y|}Y\otimes X|\,X,Y\in\cU_{i}\mbox{ for }i=\overline{0},\overline{1}\},
\end{eqnarray*}
which form two subrepresentations and $\cU\otimes\cU=\cU\odot\cU\,\oplus\,\cU\wedge\cU$.

One aim of the present paper is the characterization and properties of an important representation of $\mathfrak{osp}(m|2n)$, 
carried by harmonic
functions on the Riemannian superspace $(\mR^{m-2|2n},h_{ab})$. Therefore we need to repeat some facts related to harmonic analysis on $\mR^{p|2n}$,
which can be found in \cite{OSpHarm, MR2344451}. For a general introduction to supergeometry, see \cite{MR1701597}. In this paper we will always assume $p>2$.
The supervector $\mathbf{x}$ is defined as $\mathbf{x}=(X_1,\cdots,X_{p+2n})$, where the first $p$ variables are ordinary commuting ones and the last $2n$ are anti-commuting variables generating the Grassmann algebra $\Lambda_{2n}$. The commutation relations for both the commuting and the
Grassmann variables are captured in the relation 
$$
X_iX_j=(-1)^{[i][j]}X_jX_i ,
$$ 
where $[\cdot]$ is now understood to be the same function as before but with $m$ replaced by $p$. This defines the polynomial algebra $\cP$ generated by the variables $X_j$. For an orthosymplectic metric $h\in\mR^{(p+2n)\times(p+2n)} $, the super Laplace operator, norm squared and the Euler operator are defined by
\begin{eqnarray}
\label{DRE}
\Delta=\sum_{j,k=1}^{p+2n}\partial_{X_j}h_{jk}\partial_{X_k},\qquad R^2=\sum_{j,k=1}^{p+2n}X_jh^{jk}X_k\quad\mbox{and}\quad \mE=\sum_{j=1}^{p+2n}X_j\partial_{X_j}.
\end{eqnarray}
The operators $\Delta/2$, $R^2/2$ and $\mE+(p-2n)/2$ then generate the Lie algebra $\mathfrak{sl}(2)$. The spherical harmonics on 
$\mR^{p|2n}$ of degree $k$ are defined as
\begin{eqnarray*}
\cH_k&=&\{P\in\cP|\Delta P=0\mbox{ and }\mE P=kP\}.
\end{eqnarray*}
In this paper we consider $\cH=\bigoplus_{k=0}^\infty \cH_k$ as an $\mathfrak{osp}(p+2|2n)$ representation, but first we repeat some properties of $\cH_k$ as an $\mathfrak{osp}(p|2n)$-representation, obtained in \cite{OSpHarm}. The algebra $\mathfrak{osp}(p|2n)$ has a natural realization as differential operators on $\mR^{p|2n}$. This will be written down explicitly later but is an immediate consequence of the fact that 
\begin{eqnarray*}
%\label{PCm2n}
\cP :=\odot\,\mC^{p|2n}= \bigoplus_{k=0}^\infty \odot^k\,\mC^{p|2n}.
\end{eqnarray*}
Moreover, $R^2$ corresponds to the tensor in $\odot^2\mC^{p|2n}$ given by the metric. The harmonic polynomials are then exactly the traceless tensors. Theorem $5.1$ and $6.1$ in \cite{OSpHarm} lead to the following result.

\begin{theorem}
\label{sphHarm}
If $p-2n\not\in-2\mN$ or $p-2n=-2l$ with $k$ not in the interval $[l+2,2l+2]$, $\cH_k$ is an irreducible $\mathfrak{osp}(p|2n)$ representation, isomorphic to $K^{p|2n}_{k\epsilon_1}$.

If $p-2n=-2l$ and $l+2\le k\le 2l+2$ holds, $\cH_k$ is indecomposable as an $\mathfrak{osp}(p|2n)$-module. It has one submodule $R^{2k-2l-2}\cH_{2+2l-k}\cong K^{p|2n}_{(2+2l-k)\epsilon_1}$ and the quotient $\cH_k/R^{2k-2l-2}\cH_{2+2l-k}$ is isomorphic to 
$K^{p|2n}_{k\epsilon_1}$.
\end{theorem}

We recall some further results on $\odot\mC^{p|2n}$ obtained in \cite{OSpHarm}. If $p-2n\not\in-2\mN$, the decomposition
\begin{eqnarray}
\label{tensortracedecomp}
\odot^k\mC^{p|2n}\,&\cong&\, \odot_0^k\mC^{p|2n}\,\,\oplus\,\,\odot^{k-2}\mC^{p|2n}
\end{eqnarray}
holds with $\cH_k\cong \odot_0^k\mC^{p|2n}$ the traceless tensors. These traceless tensors can also be written as $ \circledcirc^k\mC^{p|2n}$, which denotes the Cartan product, see discussions in Sections \ref{sectenpow} and \ref{secCartan}. The embedding $\odot^{k-2}\mC^{p|2n}\hookrightarrow \odot^{k}\mC^{p|2n}$ is given by tensorial multiplication with the metric tensor in $\odot^2\mC^{p|2n}$ follows by supersymmetrization. In particular $\cH_k\cong \odot^k\mC^{p|2n}/\odot^{k-2}\mC^{p|2n}$ holds.

As already follows from Theorem \ref{sphHarm}, when $p-2n\in-2\mN$ these properties do not hold for every degree $k$. As mentioned above the representation $\cH_k\subset\odot^k\mC^{p|2n}$ of traceless tensors is not always irreducible (but still always indecomposable). The quotient representation $\odot^k\mC^{p|2n}/\odot^{k-2}\mC^{p|2n}$ is reducible for the same values $\cH_k$ is, see Theorem $5.2$ in \cite{OSpHarm}. But, for these values it also holds that $\cH_k$ is not isomorphic as an $\mathfrak{osp}(p|2n)$-representation to $\odot^k\mC^{p|2n}/\odot^{k-2}\mC^{p|2n}$. In fact the representation $\odot^k\mC^{p|2n}/\odot^{k-2}\mC^{p|2n}$ is the dual of $\cH_k$

Summarizing, in case $p-2n\in-2\mN$, the notion of traceless symmetric tensors does not necessarily correspond to the quotient of symmetric tensors with respect to symmetric tensors containing a metric term.

%%%%%%%%%%%%%%%%%%%%%%%%%%%%%%%%%%%%%%%%%%%%%%%%%%%%%%%%%%%%%%%%%%%%%%%%%%%%%%%%%%%%%%%%%%%%%%

\section{The second tensor power of the adjoint representation}

\label{sectenpow}

The adjoint action of $\osp$ on itself is given by the Lie superbracket. Since $\osp$ is a simple Lie superalgebra, 
the adjoint representation is irreducible and it has highest weight $\epsilon_1+\epsilon_2$:
\begin{eqnarray*}
\mathfrak{osp}(m|2n)\cong K^{m|2n}_{\epsilon_1+\epsilon_2}=L^{m|2n}_{2\delta_1}.
\end{eqnarray*}
In this section we will study the second tensor power of this representation. The main result is given below.
\begin{theorem}
\label{decomposition}
Consider $\mathfrak{g}=\mathfrak{osp}(m|2n;\mC)$, with $m-2n\not\in\{ 0,1,2\}$. The second tensor power of the adjoint representation decomposes into irreducible pieces as
\begin{eqnarray*}
\mathfrak{g}\otimes\mathfrak{g}&\cong&K^{m|2n}_{2\epsilon_1+2\epsilon_2}\oplus K^{m|2n}_{2\epsilon_1+\epsilon_2+\epsilon_3}\oplus K^{m|2n}_{2\epsilon_1}\oplus K^{m|2n}_{\epsilon_1+\epsilon_2+\epsilon_3+\epsilon_4}\oplus K^{m|2n}_{\epsilon_1+\epsilon_2}\oplus K^{m|2n}_0\\
&\cong&L^{m|2n}_{2\delta_1+2\delta_2}\oplus \, \,\, L^{m|2n}_{3\delta_1+\delta_2}\,\,\oplus\,\, L^{m|2n}_{\delta_1+\delta_2}\,\,\oplus\,\, L^{m|2n}_{4\delta_1}\quad\oplus\quad  L^{m|2n}_{2\delta_1}\,\,\oplus \,\, L^{m|2n}_0,
\end{eqnarray*}
where the identical representations are underneath each other. The supersymmetric tensor power $\mg\odot\mg$ corresponds to the first, third, fourth and sixth representation.
\end{theorem}
In this theorem we assumed that $m>7$ (when considering the second root system), the corresponding result for $5\le m\le 7$ can be obtained by replacing the weights by the corresponding ones in the subsequent Lemma \ref{maxvectors}.

The remainder of this section is dedicated to proving this theorem and studying the cases where the second tensor power is not completely reducible, for $M=m-2n=0,1,2$. While proving this, we also obtain useful information on and tensorial expressions for the submodules of $\mg\otimes\mg$.

\begin{lemma}
\label{maxvectors}
For $\mg=\osp$, the $\mg$-module $\mg\otimes \mg$ has at most 6 highest weight vectors, the 6 possible weights are
\begin{eqnarray*}
&&4\delta_1,\quad3\delta_1+\delta_2,\quad2\delta_1+2\delta_2,\quad2\delta_1,\quad\delta_1+\delta_2,\,\, 0\\
&&2\epsilon_1+2\epsilon_2,\quad 2\epsilon_1+\epsilon_2+\epsilon_3,\quad 2\epsilon_1,\quad \epsilon_1+\epsilon_2+\epsilon_3+\epsilon_4,\quad \epsilon_1+\epsilon_2,\,\, 0\\
&&2\epsilon_1+2\epsilon_2,\quad 2\epsilon_1+\epsilon_2+\epsilon_3,\quad 2\epsilon_1,\quad \epsilon_1+\epsilon_2+\epsilon_3+\delta_1,\quad \epsilon_1+\epsilon_2,\,\, 0\\
&&2\epsilon_1+2\epsilon_2,\quad 2\epsilon_1+\epsilon_2+\delta_1,\quad 2\epsilon_1,\quad \epsilon_1+\epsilon_2+2\delta_1,\quad \epsilon_1+\epsilon_2,\,\,  0.
\end{eqnarray*}
respectively for the standard root system, the non-standard root system if $m>7$, $m=6,7$ or $m=5$.
\end{lemma}
\begin{proof}
We give the proof explicitly for the non-standard root system and $m>7$, the other cases being completely analogous. 
As in the classical case any highest weight vector $v^+$ (of a fixed weight) in the tensor product is of the form
\begin{eqnarray*}
v^+&=&X_{\epsilon_1+\epsilon_2}\otimes A +\cdots+ Z\otimes X_{\epsilon_1+\epsilon_2},
\end{eqnarray*}
for $A,Z$ nonzero elements of $\mg$ of a certain weight.
The relation $X_\alpha v^+=0$ for all positive simple roots $\alpha$ in equation \eqref{simpleroots} leads to the condition
\begin{eqnarray*}
X_{\epsilon_1+\epsilon_2}\otimes [X_\alpha,A]+\cdots&=&0.
\end{eqnarray*}
In order for this to hold, there is either another vector $B\in\mg$ such that 
$[X_{\alpha},B]=X_{\epsilon_1+\epsilon_2}$ and $v^+$ is of the form
\begin{eqnarray*}
v^+&=&X_{\epsilon_1+\epsilon_2}\otimes A -B\otimes [X_\alpha, A]+ \cdots +Z\otimes X_{\epsilon_1+\epsilon_2},
\end{eqnarray*}
or alternatively the condition $[X_{\alpha},A]=0$ must hold.
This implies that $[X_{\alpha},A]=0$ for $\alpha\not=\epsilon_2-\epsilon_3$. This narrows the possibilities for $A$ down to $X_{\epsilon_1+\epsilon_2}$, $X_{\epsilon_1+\epsilon_3}$, $X_{\epsilon_1-\epsilon_2}$, $X_{\epsilon_3+\epsilon_4}$, the 
unique element $H$ of the Cartan subalgebra that satisfies $[X_{\alpha_j},H]=0$ for $j\not=2$ with $\alpha_j$, $j=1,\cdots,m+2n$ 
the positive simple roots in equation \eqref{simpleroots}, $X_{-\epsilon_1-\epsilon_2}$ and $X_{-\epsilon_2+\epsilon_3}$.
The last one would lead to a highest weight $\epsilon_1+\epsilon_3$ which is impossible.
\end{proof}

\begin{lemma}
\label{tensorfund}
For $m\not=2n$, the tensor product representation of the fundamental representation $\mC^{m|2n}\cong K^{m|2n}_{\epsilon_1}=L^{m|2n}_{\delta_1}$ of $\osp$ with itself has the following decomposition into irreducible representations
\begin{eqnarray*}
\mC^{m|2n}\otimes\mC^{m|2n}&\cong &K^{m|2n}_{\epsilon_1}\otimes K^{m|2n}_{\epsilon_1}=L^{m|2n}_{\delta_1}\otimes L^{m|2n}_{\delta_1}\\
&\cong& K^{m|2n}_{2\epsilon_1}\oplus K^{m|2n}_{\epsilon_1+\epsilon_1}\oplus K^{m|2n}_{0}\\
&\cong& L^{m|2n}_{\delta_1+\delta_2}\oplus L^{m|2n}_{2\delta_1}\oplus L_0^{m|2n}.
\end{eqnarray*}
The representations above each other in the last two lines are in correspondence and $\mC^{m|2n}\odot\mC^{m|2n}\cong K^{m|2n}_{2\epsilon_1}\oplus  K^{m|2n}_0$ and $\mC^{m|2n}\wedge\mC^{m|2n}\cong K^{m|2n}_{\epsilon_1+\epsilon_2}$.
\end{lemma}
\begin{proof}
The super anti-symmetric part $\mC^{m|2n}\wedge\mC^{m|2n}$ (the tensors satisfying $A^{ab}=-(-1)^{[a][b]}A^{ba}$) is clearly the adjoint representation because of equation \eqref{Liebracket}. This is an irreducible representation with highest weight $\epsilon_1+\epsilon_2$. If $m\not=2n$, the super symmetric part $\mC^{m|2n}\odot\mC^{m|2n}$ decomposes as $ K^{m|2n}_{2\epsilon_1}\oplus K^{m|2n}_{0}$ as has been proven in \cite{OSpHarm}, or follows from Equation \eqref{tensortracedecomp}. The corresponding highest weights for the standard root system can then be obtained from Theorem 3 in \cite{Tensor}.
\end{proof}

As was mentioned at the end of Section \ref{preli}, the trivial representation $K^{m|2n}_0$ inside $\mC^{m|2n}\otimes\mC^{m|2n}$ is given by the tensor $g^{ab}$. If $M=m-2n\not=0$, the two other components correspond to traceless tensors. For the super skew tensors, this is a consequence of equation \eqref{traces}. The statement for the super symmetric tensors follows from the fact that they can be written as
\begin{eqnarray*}
A^{ab}&=&\left(A^{ab}-\frac{A_c{}^c}{M}g^{ab}\right)+\frac{A_c{}^c}{M}g^{ab},
\end{eqnarray*}
if $M\not=0$.

\begin{remark}
\label{relativeCartan}
{\rm Lemma \ref{tensorfund} already shows that the notion of Cartan product depends on the choice of root system. For the standard root system, the Cartan product is $L^{m|2n}_{2\delta_1}=K^{m|2n}_{\epsilon_1+\epsilon_2}$. This actually is inside the super anti-symmetric part of $\mC^{m|2n}\otimes\mC^{m|2n}$, which would be very unnatural for a Cartan product. However, $L_{2\delta_1}^{m|2n}$ is inside the super symmetric part of $\mC^{2n|m}\otimes\mC^{2n|m}$, which corresponds to the $\mathfrak{spo}(2n|m)$-interpretation. The Cartan product with respect to the second root system is $K^{m|2n}_{2\epsilon_1}=L^{m|2n}_{\delta_1+\delta_2}$ and is inside the super symmetric part. Therefore we have a different Cartan product as an $\mathfrak{spo}(2n|m)$-representation and as an $\mathfrak{osp}(m|2n)$-representation.}
\end{remark}

\begin{lemma}
\label{embeddingCC}
There is an injective homomorphism $\phi$ from 
$\mC^{m|2n}\otimes \mC^{m|2n}$ into $\mg\otimes\mg$, 
with $\mg=\osp$. This implies that for $m-2n\not=0$,
\begin{eqnarray*}
K^{m|2n}_{2\epsilon_1}\oplus K^{m|2n}_0\subset \mg \odot\mg\quad &\mbox{and}&\quad K^{m|2n}_{\epsilon_1+\epsilon_2}\subset \mg\wedge\mg
\end{eqnarray*}
hold.
\end{lemma}
\begin{proof}
Firstly, we construct the $\mathfrak{osp}(m|2n)$-module homomorphism
\begin{eqnarray*}
\phi_1:\otimes^2\mC^{m|2n}&\to&\otimes^4\mC^{m|2n}:\\
\phi_1(X)=T&&T^{abcd}=g^{bc}X^{ad}.
\end{eqnarray*}
The embedding of $\mC^{m|2n}\otimes\mC^{m|2n}$ in $\mg\otimes\mg$ is then given by the composition of $\phi_1$ with $\mg$-invariant projection $\otimes^4\mC^{m|2n}\to\left(\mC^{m|2n}\wedge\mC^{m|2n}\right)\otimes\left(\mC^{m|2n}\wedge\mC^{m|2n}\right)\cong\mg\otimes\mg$. So the embedding is given by $\phi(X)=V$, with
\begin{eqnarray*}
V^{abcd}&=&\frac{1}{4}\left(g^{bc}X^{ad}-(-1)^{[a][b]}g^{ac}X^{bd}-(-1)^{[c][d]}g^{bd}X^{ac}+(-1)^{[a][b]+[c][d]}g^{ad}X^{bc}\right).
\end{eqnarray*}
It is clear that $\phi$ maps $\mC^{m|2n}\wedge\mC^{m|2n}$ into $\mg\wedge\mg$ and $\mC^{m|2n}\odot\mC^{m|2n}$ into $\mg\odot\mg$. The rest of the lemma then follows immediately from Lemma \ref{tensorfund}.
\end{proof}

In particular, the realization of the trivial representation $K^{m|2n}_0$ inside $\mg\otimes\mg$ is given by the tensor
\begin{eqnarray}
\label{tracepart}
V^{abcd}=\frac{1}{2}\left(g^{bc}g^{ad}-(-1)^{[a][b]}g^{ac}g^{bd}\right),
\end{eqnarray}
called the total trace part of $\mg\otimes\mg$. When the projection of $\otimes\mg$ onto $\cU(\mg)$ is considered this corresponds to the quadratic Casimir operator.

\begin{lemma}
\label{subA}
If $M=m-2n\not\in\{1,2\}$, the second tensor power of the adjoint representation of $\mg=\osp$ has the following $\mg$-module decomposition:
\begin{eqnarray*}
\mg\otimes \mg&\cong& \mC^{m|2n}\otimes\mC^{m|2n}\quad\oplus\quad \cA,
\end{eqnarray*}
with $\cA$ the tensors $T^{abcd}$ in $\mg\otimes\mg$ that satisfy the relation $\sum_{b}T^a{}_b{}^{bd}=0$.
\end{lemma}
\begin{proof}
We denote the trace map by 
\begin{eqnarray*}
p:\mg\otimes\mg\to\mC^{m|2n}\otimes\mC^{m|2n},\quad p(T)=Y\quad& \mbox{with}&\quad Y^{ad}=\sum_bT^a{}_b{}^{bd}.
\end{eqnarray*}
This clearly is a $\mg$-module morphism. The composition of $p$ with $\phi$ from Lemma \ref{embeddingCC} is denoted by $\chi=\phi\circ p:\mg\otimes\mg\to\mg\otimes\mg$. The two spaces Im$(\chi)$ and Ker$(\chi)$ are subrepresentations of $\mg\otimes\mg$. 

It is clear that Ker$(\chi)= $ Ker$(p)=\cA$ holds. Next we prove that Im$(\chi)=$ Im$(\phi)$ holds. This is a consequence of the fact that $p$ is surjective, which we prove by calculating the composition $p\circ\phi$ on $\mC^{m|2n}\otimes\mC^{m|2n}$:
\begin{eqnarray}
\label{pphi}
X^{ab}&\to&\frac{1}{4}\left((M-2)X^{ab}+g^{ab}\sum_cX_{c}{}^c\right),
\end{eqnarray}
which follows from equation \eqref{traces}. This implies that for $M=m-2n\not\in\{0,2\}$, $K^{m|2n}_{2\epsilon_1}\oplus K^{m|2n}_{\epsilon_1+\epsilon_2}\subset$ Im$(p)$. Substituting $g^{ab}$ for $X^{ab}$ in equation \eqref{pphi} implies that the total trace part is mapped under $p\circ\phi$ to
\begin{eqnarray}
\label{tracepphi}
g^{ab}&\to&\frac{1}{2}(M-1)g^{ab},
\end{eqnarray}
so for $M=m-2n\not\in\{0,1\}$, $K^{m|2n}_0\subset$ Im$(p)$, which leads to $\mC^{m|2n}\otimes \mC^{m|2n}=$ Im$(p)$ provided $m-2n\not\in\{0,1,2\}$.

The case left is $M=0$. The surjectivity of mapping \eqref{pphi} follows from injectivity. This injectivity can be checked immediately for traceless tensors. If the mapping would be zero for a tensor which is not traceless than in particular the trace of the right-hand side of equation \eqref{pphi} should be zero, which is never the case for such a tensor if $M=0$.

The calculations above also imply that Ker$(\chi)\cap$Im$(\chi)=$ Ker$(p)\cap$Im$(\phi)=\{0\}$ and therefore $\mg\otimes\mg=$Im$(\chi)\oplus$Ker$(\chi)$, which completes the proof. \end{proof}

Now we get to the proof of Theorem \ref{decomposition}.
\begin{proof}
Consider the subrepresentation $\cA\subset\mg\otimes\mg$ of traceless tensors from Lemma \ref{subA}. First we remark that $\cA\cap (\mathfrak{g}\wedge\mathfrak{g})\not=0$, which is a consequence of the corresponding claim for 
$\mathfrak{so}(m)\subset\mg$. We call this subrepresentation $\cA_1$. The representation 
$\cA_2=\cA\cap (\mathfrak{g}\odot\mathfrak{g})$ is nonzero for the same reason.

We define the $\mg$-module morphism $q:\cA_2\to\cA_2$ given by
\begin{eqnarray}
\label{FormCartan2}
V^{abcd}&\to&\frac{1}{3}\left(V^{abcd}+(-1)^{[a]([b]+[c])}V^{bcad}-(-1)^{[b][c]}V^{acbd}\right)
\end{eqnarray}
for $V\in\cA_2\subset \mg\odot\mg$. This is a projection ($q^2=q$) and we obtain that the kernel and the 
image of $q$, $\cA_2^{im}$ and $\cA^{ker}_2$, satisfy $\cA_2=\cA_2^{im}\oplus\cA_2^{ker}$ as $\mg$-representations. The 
non-emptyness of these subrepresentations follows from the classical $\mathfrak{sp}(2n)$-case. 

Therefore by considering Lemma \ref{tensorfund} and Lemma \ref{subA} we have proven that as a $\mathfrak{g}$-representation, the decomposition
\begin{eqnarray*}
\mg\otimes \mg&\cong& K^{m|2n}_{2\epsilon_1}\oplus K^{m|2n}_{\epsilon_1+\epsilon_2}\oplus K^{m|2n}_0\oplus \cA_1\oplus \cA_2^{im}\oplus\cA_2^{ker}
\end{eqnarray*}
holds. So $\mathfrak{g}\otimes \mathfrak{g}$ has at least 6 disjoint subrepresentation and by Lemma \ref{maxvectors} at most 6 highest weight vectors. Therefore, Corollary 1 in \cite{Tensor} can be used to conclude that $\mathfrak{g}\otimes\mathfrak{g}$ is completely reducible and has exactly the 6 weights from Lemma \ref{maxvectors} appearing as highest weight vectors.
\end{proof}

Then it is easily verified that $\cA_1=K^{m|2n}_{2\epsilon_1+\epsilon_2+\epsilon_3}$, $\cA^{ker}_2=K^{m|2n}_{2\epsilon_1+2\epsilon_2}$ and $\cA^{im}_2=K^{m|2n}_{\epsilon_1+\epsilon_2+\epsilon_3+\epsilon_4}$.

\begin{remark}
\label{qsymm}
{\rm The $\mg$-module morphism $q$ in equation \eqref{FormCartan2} can be defined on $\mg\odot\mg$ as well instead of on $\cA_2$, the traceless tensors in $\mg\odot\mg$. Since the image of $q$ is always tracefree ($p\circ q=0$), we find that $K^{m|2n}_{\epsilon_1+\epsilon_2+\epsilon_3+\epsilon_4}$ is also equal to the image of $q$ acting on $\mg\odot\mg$,
\begin{eqnarray*}
K^{m|2n}_{\epsilon_1+\epsilon_2+\epsilon_3+\epsilon_4}&=&\{\frac{1}{3}(V^{abcd}+(-1)^{[a]([b]+[c])}V^{bcad}-(-1)^{[b][c]}V^{acbd})|V\in\mg\odot\mg\}.
\end{eqnarray*}
The kernel of $q$ is related to the representation $K^{m|2n}_{2\epsilon_1+2\epsilon_2}$:
\begin{eqnarray*}
K^{m|2n}_{2\epsilon_1+2\epsilon_2}=\{V\in\mg\odot\mg|V^{abcd}+(-1)^{[a]([b]+[c])}V^{bcad}-(-1)^{[b][c]}V^{acbd}=0\mbox{ and $V$ is traceless}\}.
\end{eqnarray*}}
\end{remark}

According to the standard choice of positive roots, the Cartan product inside $\mg\otimes\mg$ is given by $L^{m|2n}_{4\delta_1}=K^{m|2n}_{\epsilon_1+\epsilon_2+\epsilon_3+\epsilon_4}\subset\mg\odot\mg$, it is part of the completely traceless part inside $\mg\odot\mg$, denoted by $\cA_2$ in the proof of Theorem \ref{decomposition}. According to the second choice of positive roots, the Cartan product inside $\mg\otimes\mg$ is given by $K^{m|2n}_{2\epsilon_1+2\epsilon_2}=L^{m|2n}_{2\delta_1+2\delta_2}\subset\mg\odot\mg$, it is also a component of the completely traceless part inside $\mg\odot\mg$, in fact $\cA_2=K^{m|2n}_{2\epsilon_1+2\epsilon_2}\oplus K^{m|2n}_{\epsilon_1+\epsilon_2+\epsilon_3+\epsilon_4}$ if $M\not\in\{1,2\}$. 

For $\mg\cong\osp\cong\mathfrak{spo}(2n|m)$ we introduce the notations
\begin{eqnarray*}
\mg\topa\mg &=& \mathfrak{osp}(m|2n)\circledcirc\osp\cong K^{m|2n}_{2\epsilon_1+2\epsilon_2},\mbox{ if }M\not\in\{1,2\}\\
\mg\topb\mg&=&\mathfrak{spo}(2n|m)\circledcirc\mathfrak{spo}(2n|m)\cong L^{m|2n}_{4\delta_1}.
\end{eqnarray*}
The two Cartan products are therefore given by
\begin{eqnarray*}
\mathfrak{g}\topa\mathfrak{g}\cong K^{m|2n}_{2\epsilon_1+2\epsilon_2}\cong L^{m|2n}_{2\delta_1+2\delta_2}&\quad\mbox{and}\quad&\mathfrak{g}\topb\mathfrak{g}\cong L^{m|2n}_{4\delta_1}\cong K^{m|2n}_{\epsilon_1+\epsilon_2+\epsilon_3+\epsilon_4}.
\end{eqnarray*}
To avoid confusion we will use the notations $\mg\topa\mg$ and $\mg\topb\mg$ rather than $\osp\circledcirc\osp$ and $\mathfrak{spo}(2n|m)\circledcirc\mathfrak{spo}(2n|m)$.

\begin{remark}
\label{Cartandiff}
{\rm
There is an important difference between the two Cartan products. For $\mg=\osp\cong\mathfrak{spo}(2n|m)$ with $m-2n\not\in\{0,1,2\}$, we denote $\Pi^{m|2n}_1$ and $\Pi^{m|2n}_2$ the unique $\mg$-invariant projections from $\mg\otimes\mg$ onto respectively $\mg\topa\mg$ and $\mg\topb\mg$. There are $X\in\mathfrak{so}(m)\otimes\mathfrak{so}(m)$ such that
\begin{eqnarray*}
\Pi_1^{m|0}(X)&\not=&\Pi_1^{m|2n}(X)
\end{eqnarray*}
holds, because the trace parts that need to be subtracted in the $\osp$-case are given by the tensor $g^{ab}$ and not the metric-part of $\mathfrak{so}(m)$. In other words the projection onto the Cartan product of $\osp\otimes\mathfrak{osp}(m|2n)$ of an element in $\mathfrak{so}(m)\otimes\mathfrak{so}(m)$ is not equal to its projection onto the Cartan product $\mathfrak{so}(m)\circledcirc\mathfrak{so}(m)$.

From the form of the second Cartan product in Remark \ref{qsymm}, it follows that
\begin{eqnarray*}
\Pi_1^{0|2n}(Y)&=&\Pi_1^{m|2n}(Y)\qquad\mbox{holds for all}\quad Y\in\mathfrak{sp}(2n)\otimes\mathfrak{sp}(2n).
\end{eqnarray*}
This different behavior will be important for the Joseph-like ideals and the corresponding minimal representations.}
\end{remark}

Now we consider the structure of the second power of the adjoint representation for the exceptional cases, $M$ equal to $0$, $1$ or $2$.
\begin{theorem}
\label{cases012}
In the cases $M=m-2n\in\{0,1,2\}$ the representation $\mg\otimes\mg$ is not completely reducible for $\mg=\osp$. If $M=0$, the decomposition of $\mg$-representations
\begin{eqnarray*}
\mg\otimes\mg&\cong&K^{m|2n}_{2\epsilon_1+2\epsilon_2}\oplus K^{m|2n}_{2\epsilon_1+\epsilon_2+\epsilon_3}\oplus K^{m|2n}_{\epsilon_1+\epsilon_2+\epsilon_3+\epsilon_4}\oplus K^{m|2n}_{\epsilon_1+\epsilon_2}\,\oplus \,\,\left(\mC^{m|2n}\odot\mC^{m|2n}\right).\end{eqnarray*}
holds, where $\mC^{m|2n}\odot\mC^{m|2n}$ is not completely reducible. If $M\in\{1,2\}$, the decomposition of $\mg$-representations
\begin{eqnarray*}
\mg\otimes\mg&\cong&  K^{m|2n}_{\epsilon_1+\epsilon_2+\epsilon_3+\epsilon_4}\,\oplus\,  \gamma\,\oplus \,\mg\wedge\mg\end{eqnarray*}
holds where $ \gamma$ is the representation corresponding to the tensors
\begin{eqnarray*}
 \gamma&=&\{V\in\mg\odot\mg|V^{abcd}+(-1)^{[a]([b]+[c])}V^{bcad}-(-1)^{[b][c]}V^{acbd}=0\},
\end{eqnarray*}
which is not completely reducible and the representation $K^{m|2n}_{\epsilon_1+\epsilon_2+\epsilon_3+\epsilon_4}$ is given as in Remark \ref{qsymm}. 
\end{theorem}
\begin{proof}
Firstly we consider the case $M=0$, then Lemma \ref{embeddingCC} and Lemma \ref{subA} hold. The proof that $\cA=\cA_1\oplus\cA^{im}_2\oplus\cA_2^{ker}$ holds is then completed similarly to the proof of Theorem \ref{decomposition}.

For the case $M=1,2$, we consider the morphism $\chi$ from the proof of Lemma \ref{subA} restricted to $\mg\odot\mg$. It is clear from the proof of Lemma \ref{subA} that for these dimensions Ker$(\chi)\cap$Im$(\chi)\not=0$, in particular $K^{m|2n}_{2\epsilon_1}\subset$ Ker$(\chi)\cap$Im$(\chi)$ if $M=2$ and $K^{m|2n}_{0}\subset$Ker$(\chi)\cap$Im$(\chi)$ if $M=1$.

Therefore Ker$(\chi)+$Im$(\chi)\not=\mg\odot\mg$. We also consider the morphism $q$ defined in equation \eqref{FormCartan2} but now defined on the entire space $\mg\odot\mg$, as in Remark \ref{qsymm}. This leads to 
\begin{eqnarray*}
\mg\otimes\mg&=&\mbox{Im}(q)\oplus  \gamma\oplus \mg\wedge\mg.
\end{eqnarray*}
with Im$(q)\subset$ Ker$(\chi)$ and Im$(\chi)\subset$ Ker$(q)=\,\gamma$. The module Im$(q)$ must contain at least one highest weight vector. The module $ \gamma$ has at least three, the vector $X_{\epsilon_1+\epsilon_2}\odot X_{\epsilon_1+\epsilon_2}$ and the two highest weight vectors inside $\mC^{m|2n}\odot\mC^{m|2n}$. Finally the module $\mg\wedge\mg$ has at least two, $X_{\epsilon_1+\epsilon_2}\otimes X_{\epsilon_1+\epsilon_3}-X_{\epsilon_1+\epsilon_3}\otimes X_{\epsilon_1+\epsilon_2}$ and the highest weight vector of $\mC^{m|2n}\wedge\mC^{m|2n}$. Lemma \ref{maxvectors} implies that these are all of the highest weight vectors in $\mg\otimes\mg$, so in particular $ \gamma$ has exactly three highest weight vectors.

Next we prove that $ \gamma$ has a subrepresentation that contains the three highest weight vectors, which implies that $ \gamma$ is not completely reducible. This subrepresentation is given by
\begin{eqnarray*}
\left(\mbox{Ker}(\chi)+\mbox{Im}(\chi)\right)\cap  \gamma.
\end{eqnarray*}
If this were equal to $ \gamma$ it would imply that $\left(\mbox{Ker}(\chi)+\mbox{Im}(\chi)\right)\cap\mbox{Im}(q)$ is not equal to Im$(q)$ which is a contradiction with $\mbox{Im}(q)\subset\mbox{Ker}(\chi)$. The two highest weight vectors of $\mC^{m|2n}\odot\mC^{m|2n}$ are contained in this subrepresentation as they are inside Im$(\chi)$ and from the classical case it follows that also $X_{\epsilon_1+\epsilon_2}\otimes X_{\epsilon_1+\epsilon_2}$ is contained. The identification Im$(q)\cong K^{m|2n}_{\epsilon_1+\epsilon_2+\epsilon_3+\epsilon_4}$ then follows from the fact that the other five weights are already assigned to the other representations.
\end{proof}

Therefore we obtain that $\mg\topa\mg\cong K^{m|2n}_{2\epsilon_1+2\epsilon_2}\subset\mg\otimes\mg$ is still well-defined if $M=0$ and $\mg\topb\mg\cong K^{m|2n}_{\epsilon_1+\epsilon_2+\epsilon_3+\epsilon_4}\subset\mg\otimes\mg$ is still well-defined if $M=0,1,2$ (as in Remark \ref{qsymm}).

%%%%%%%%%%%%%%%%%%%%%%%%%%%%%%%%%%%%%%%%%%%%%%%%%%%%%%%%%%%%%%%%%%%%%%%%%%%%%%%%%%%%%%%%%%
\section{The Cartan product}
\label{secCartan}
One of the results in \cite{MR2130630} is that for $V$ a finite dimensional irreducible $\mathfrak{sl}(n)$-module, the property
\begin{eqnarray}
\label{formCartan}
\circledcirc^kV&=&\left(\circledcirc^{k-1}V\right)\otimes V\cap V\otimes\left(\circledcirc^{k-1}V\right)
\end{eqnarray}
holds. In this section we prove that this property can be extended to an arbitrary semisimple Lie algebra. With an extra assumption we prove that it also holds for semisimple Lie superalgebras. These results are stated in Theorem \ref{thmCartan} and generalized in Corollary \ref{claimEast} which proves the conjecture made in \cite{MR2130630}. Because of examples, there is reason to believe that property \eqref{formCartan} can still hold without that extra assumption, but with a slight reformulation. For $V$ a finite dimensional irreducible representation with highest weight vector $v_+$ of a semisimple Lie superalgebra $\mg$, the question is whether the property
\begin{eqnarray}
\label{herformCartan}
\cU(\mg)\cdot v_+^{\otimes k}&=&\left(\cU(\mg)\cdot  v_+^{\otimes k-1}\,\otimes \,V\right)\cap \left(V\,\otimes\,\cU(\mg)\cdot  v_+^{\otimes k-1}\right)
\end{eqnarray}
holds for $k>2$. 

In case $V=\mC^{m|2n}$ and $\mg=\osp$, equation \eqref{herformCartan} can be obtained from the results in \cite{OSpHarm}, repeated in Section \ref{preli}. So in this case equation \eqref{herformCartan} holds even when $\odot^k V$ is not completely reducible, which happened sometimes in case $m-2n\in-2\mN$.

In case $V=\mg$ with $\mg=\osp$ for both positive root systems there is a direct application of the results to connect the Joseph-like ideals to annihilator ideals of representations further in this paper. In case $\cU(\mg)\cdot v_+^{\otimes k}$ is irreducible, statement \eqref{herformCartan} reduces to the classical one \eqref{formCartan}. Even though it is clear from the aforementioned example, that this property is not always necessary to obtain equation \eqref{herformCartan}, we only manage to prove this result in cases where it turns out that $\cU(\mg)\cdot v_+^{\otimes k}$ is irreducible, in Theorem \ref{Cartanosp}.

In this section we first consider an arbitrary semisimple Lie superalgebra $\mg$ with a fixed positive root system and an arbitrary irreducible finite dimensional representation $V$. The set of positive roots is denoted by $\Delta^+$ and for each such root we fix the positive and negative root vectors $X_\alpha$ and $Y_\alpha$. The corresponding triangular decomposition is given by $\mg=\mathfrak{n}^-+\mathfrak{h}+\mathfrak{n}$.

First we show that the proposed identity is equivalent to other formulations. Therefore we recall that $V$ has a non-degenerate hermitian form $\langle\cdot,\cdot\rangle$ such that
\begin{eqnarray*}
\langle X_\alpha u,w\rangle&=&(-1)^{|\alpha||u|}\langle u, Y_\alpha w\rangle
\end{eqnarray*}
holds, see Lemma 1 in \cite{Tensor}. This is a contravariant hermitian form. Furthermore different weight spaces are orthogonal to each other. This property immediately extends to the tensor power representations $\otimes^k V$ where the form is defined iteratively by $\langle v\otimes a,u\otimes b\rangle=(-1)^{|a||u|}\langle v,u\rangle\langle a,b\rangle$ for $a,b\in\otimes^{k-1}V$ and $v,u\in V$. We fix notations 
\begin{eqnarray}
\label{defBk}
\beta_2=\cU(\mg)\cdot(v_+\otimes v_+)&\quad\mbox{and}\quad& \beta_k=\beta_{k-1}\otimes V\cap V\otimes \beta_{k-1}\mbox{ for }k>2\\
\nonumber
I_k=\beta_{k}^\perp&&\mbox{with respect to $\langle\cdot,\cdot\rangle$ for }k\ge 2
\end{eqnarray}
for subrepresentations in $\otimes^kV$. Equation \eqref{herformCartan} can then be restated as $\beta_k=\cU(\mg)\cdot v_+^{\otimes k}$ for $k>2$.

It follows immediately by induction that the properties 
\begin{equation}
\label{Bsymm}
\cU(\mg)\cdot v_+^{\otimes k}\,\subset\,\beta_k\,\subset\, \odot^k V
\end{equation} 
hold and that 
\begin{equation}
\label{altdefIk}
I_{k}=I_{k-1}\otimes V+ V\otimes I_{k-1},
\end{equation}
which can be used as an alternative definition for the representation $I_k$.
\begin{lemma}
\label{Lemequiv}
For $\beta_k$ and $I_k$ as defined in equation \eqref{defBk}, the following statements are equivalent:
\begin{eqnarray*}
&(1)& \mbox{ The representation $\beta_k$ is generated by the highest weight vector of $\otimes^kV$: }\beta_k=\cU(\mg)\cdot v_+^{\otimes k}.\\
&(2)& \mbox{ The representation $\otimes^kV/I_k$ has only one highest weight vector, $v_+^{\otimes k}+I_k$.}\\
&(3) &\mbox{ Any representation $K\subset\otimes^k V$ that properly contains $I_k$, also contains $v_+^{\otimes k}$.}\\
&(4)& \mbox{ For any vector $x\in\otimes^k V$, with $x\not\in I_k$, $v_+^{\otimes k}\in\cU(\mg)\cdot x$ holds.}
\end{eqnarray*}
\end{lemma}
\begin{proof}
First we prove the equivalency between the first and the second statement. Because $\beta_k$ and $I_k$ satisfy $\beta_k^\perp=I_k$ with respect to a non-degenerate contravariant hermitian form, it follows that the codimension of $\mathfrak{n}^-\cdot\beta_k$ in $\beta_k$ is equal to the dimension of the space of highest weight vectors (vectors annihilated by $\mn$) in $\otimes^kV/I_k$. The first equivalency is a consequence of this.

Each representation $K\supset I_k$ has a vector $x$ of maximal weight that is not inside $I_k$. Positive root vectors acting on this vector $x$ must map it to $I_k$, this implies that $x+I_k$ is a highest weight vector in $\otimes^kV/I_k$. If property two holds, this vector $x$ has to correspond to $v_+^{\otimes k}$ and property three holds. Similarly each highest weight vector in $\otimes^kV/I_k$ can be used to construct a representation $K\supset I_k$, so property three implies property two.

The equivalence of the third and fourth property is straightforward.\end{proof}

\begin{lemma}
\label{Lemequiv2}
Consider $\beta_k$ as defined in equation \eqref{defBk}. If $\beta_{k-1}=\cU(\mg)\cdot v_+^{\otimes k-1}$ holds, then $\beta_k$ satisfies $\beta_k=v_+^{\otimes k}+(\mathfrak{n}^-\cdot\otimes^kV\cap \beta_k)$.
\end{lemma}
\begin{proof}
If $\beta_{k-1}=\cU(\mg)\cdot v_+^{\otimes k-1}=\cU(\mn^-)\cdot v_+^{\otimes k-1}$ holds, the representation $\beta_{k-1}\otimes V$ is equal to $\cU(\mn^-)\cdot\left(v_+^{\otimes k-1}\otimes V\right)$, as follows from the Leibniz rule of tensor product representations.

Now we assume that $x\in\beta_k$ is not inside $\mn^-\cdot\otimes^k V$. Since it is an element of $\cU(\mn^-)\cdot\left(v_+^{\otimes k-1}\otimes V\right)$
it can be written as
\[x=v_+^{\otimes k-1}\otimes u+x_1\]
with $u\in V$ and $x_1\in \mathfrak{n}^-\cU(\mn^-)\cdot(v_+^{\otimes k-1}\otimes V)$. Because $x$ is inside $\odot^k V$ by equation \eqref{Bsymm}, $x_1$ must contain a term $u\otimes v_+^{\otimes k-1}$. However $x_1\in\cU(\mn^-)\cdot\left(v_+^{\otimes k-1}\otimes V\right)$ which implies that this term must come from an element in $\cU(\mn^-)\cdot v_+^{\otimes k}$. Elements inside $\cU(\mn^-)\cdot v_+^{\otimes k}$ are symmetric on their own, so can never be used to compensate for the lack of symmetry of other terms. 

This shows that $u$ must be proportional to $v_+$ which completes the proof of the lemma.
\end{proof}

When $\cU(\mg)\cdot v_+^{\otimes k}$ is irreducible it is isomorphic to $M_{k\lambda}$ if $V\cong M_\lambda$. If $\cU(\mg)\cdot v_+^{\otimes k}$ also has a complement representation $W$, we can use the notation $\circledcirc^k V=\cU(\mg)\cdot v_+^{\otimes k}\subset \otimes^k V$. We say that in those cases $\circledcirc^k V$ is well defined and the property $\otimes^k V=\circledcirc^k V\oplus W$ holds.

\begin{theorem}
\label{thmCartan}
Consider a semisimple Lie superalgebra $\mg$ and a finite dimensional representation $V$ with highest weight vector $v_+$. If $\beta_j$, defined in equation \eqref{defBk}, has a complement representation inside $\otimes^j V$ for $2\le j\le k$, $\circledcirc^j V\subset\otimes^j V$ is well-defined and the properties
\begin{eqnarray*}
\beta_j&=&\cU(\mg)\cdot v_+^{\otimes j}=\circledcirc^j V\quad\mbox{ and }\\
\circledcirc^jV&=&\left(\circledcirc^{j-1}V\right)\otimes V\cap V\otimes\left(\circledcirc^{j-1}V\right)
\end{eqnarray*}
hold for $j\le k$. In case $\mg$ is a semisimple Lie algebra, the extra condition on $\beta_j$ is not necessary.\end{theorem}
\begin{proof}
By definition $\beta_2=\cU(\mg)\cdot v_+^{\otimes 2}$. We proceed by induction and assume $\beta_{j-1}=\cU(\mg)\cdot v_+^{\otimes j-1}$. Since $\beta_j$ has a complement representation, $(\mathfrak{n}^-\cdot\otimes^jV\cap \beta_j)=\mn^-\cdot\beta_j$ and therefore Lemma \ref{Lemequiv2} implies that $\beta_{j}=\cU(\mg)\cdot v_+^{\otimes j}$.

To complete the proof we show that the fact that $\beta_j=\cU(\mg)\cdot v_+^{\otimes j}$ has a complement representation implies that it is irreducible. If $\cU(\mg)\cdot v_+^{\otimes j}$ would contain a highest weight vector $x_+$ other than $v_+^{\otimes j}$, this vector would be orthogonal with respect to all other elements of $\cU(\mg)\cdot v_+^{\otimes j}$ for the contravariant hermitian form. Since the form is non-degenerate there exists a vector $u$ in the complement representation of $\cU(\mg)\cdot v_+^{\otimes j}$ such that $\langle x_+,u\rangle\not=0$. By using the contravariance of the hermitian form this implies that $v_+^{\otimes j}\in\cU(\mg)\cdot u$, which is a contradiction. Since $\cU(\mg)\cdot v_+^{\otimes j}$ is generated by a highest weight vector, the lack of other highest weight vectors proves its irreducibility.
\end{proof}

In case $\mg=\mathfrak{osp}(2|2n)$ or $\mg=\mathfrak{gl}(p|q)$, $\mg$ has large classes of star representations for which the tensor power is always completely reducible, see \cite{MR1773773, MR0424886}, so where the extra condition on $\beta_j$ from Theorem \ref{thmCartan} is not needed.

The method in Theorem \ref{thmCartan} can easily be extended to prove the following claim which was made by Eastwood underneath Corollary 1 in \cite{MR2130630}.
\begin{corollary}
\label{claimEast}
For any irreducible finite dimensional representation $V$ of a semisimple Lie algebra, the property 
\begin{eqnarray*}
\left[\left((\circledcirc^pV)\circledcirc (\circledcirc^q V)\right)\,\otimes\, \circledcirc^r V\right]\cap\left[(\circledcirc^pV)\,\otimes\, \left( (\circledcirc^q V)\circledcirc(\circledcirc^r V)\right)\right]&=&\circledcirc^{p+q+r}V
\end{eqnarray*}
holds for $p,q,r$ strictly positive natural numbers.
\end{corollary}
\begin{proof}
If $r < p+q$ this can be proved by considering \[\cU(\mn^-)\cdot\left(v_+^{p+q}\otimes \circledcirc^r V\right)\cap\odot^{p+q+r}V\]
with the same techniques as in the proof of Lemma \ref{Lemequiv2} and using complete reducibility. If $r\ge p+q$ then $p<q+r$ and the proof can be done that way.
\end{proof}

In order to prove equation \eqref{herformCartan} or \eqref{formCartan} in the cases of our interest we need another lemma.
\begin{lemma}
\label{betaCasimir}
The representation $\beta_k$ is contained in the subrepresentation of $\otimes^kV$ consisting of those elements that are eigenvectors of the quadratic Casimir operator with the same eigenvalue as $v_+^{\otimes k}$.
\end{lemma}
\begin{proof}
By definition this holds for $\beta_2$. The statement for $k>2$ follows from the fact that the restriction to each two positions in the tensor product of $\otimes^k V$ of an element of $\beta_k$ is contained in $\beta_2$ and the fact that the Leibniz rule spreads the quadratic Casimir operator over at most two positions.
\end{proof}

\begin{theorem}
\label{Cartanosp}
Equation \eqref{herformCartan} holds for $\mg=\mathfrak{osp}(m|2n)$ and $V$ the adjoint representation with
\begin{itemize}
\item the positive root system corresponding to $\mathfrak{spo}(2n|m)$,
\item the positive root system corresponding to $\osp$ if $m-2n> 2$,
\end{itemize}
where the root systems are given in equations \eqref{simpleroots2} and \eqref{simpleroots}. In both cases the statement reduces to equation \eqref{formCartan}.
\end{theorem}
\begin{proof}
First of all we note that since the adjoint representation of $\mg=\osp$ corresponds to the representation on $\mC^{m|2n}\wedge\mC^{m|2n}$ it naturally extends to a $\mathfrak{gl}(m|2n)$-representation. The same holds for the tensor powers $\otimes^j\mg$.

For the first case the representation $\beta_j$ corresponds to the space
\begin{eqnarray*}
\beta_j&=&\{T\in\odot^j\mg|2T^{abcd\cdots ef}=(-1)^{[a]([b]+[c])}T^{bcad\cdots ef}-(-1)^{[b][c]}T^{acbd\cdots ef}\},
\end{eqnarray*}
see Remark \ref{qsymm}. This is a $\mathfrak{gl}(m|2n)$-tensor module (star representation, see \cite{MR1773773, MR0424886}) and since tensor products of irreducible $\mathfrak{gl}(m|2n)$-tensor modules are completely reducible, it follows that $\beta_j$ has a complement representation as a $\mathfrak{gl}(m|2n)$-representation and hence as a representation of $\osp\subset\mathfrak{gl}(m|2n)$ too. The result then follows immediately from Theorem \ref{thmCartan}.

For the second case, the representation $\beta_j$ corresponds to the subrepresentation of traceless tensors inside the space
\begin{eqnarray*}
\gamma_j&=&\{T\in\odot^j\mg|T^{abcd\cdots ef}+(-1)^{[a]([b]+[c])}T^{bcad\cdots ef}-(-1)^{[b][c]}T^{acbd\cdots ef}=0\}.
\end{eqnarray*}
For the same reason as above $\gamma_j$ has a complement representation inside $\otimes^j \mg$. It follows easily from the concept of taking traces that all possible $\osp$-highest weight vectors inside $\gamma_j$ are of weight
\begin{eqnarray*}
(j-p)\epsilon_1+(j-q)\epsilon_2
\end{eqnarray*}
with $p+q$ even and $p,q\le j$. It can be checked easily, see \cite{MR0546778}, that all of these weights lead to a different eigenvalue of the quadratic Casimir operator exactly when $m-2n>2$. The representation $\gamma_j$ decomposes into generalized eigenspaces of the quadratic Casimir operator, mutually orthogonal with respect to the contravariant hermitian form. Since the one containing $v_+^{\otimes j}$ has only one highest weight vector and has a non-degenerate contravariant hermitian form, it follows straightforwardly that it is an irreducible highest weight representation. This has to be $\beta_j$ by Lemma \ref{betaCasimir}, so $\beta_j=\cU(\mg)\cdot v_+^{\otimes j}=\circledcirc^j \mg$. 
\end{proof}

%%%%%%%%%%%%%%%%%%%%%%%%%%%%%%%%%%%%%%%%%%%%%%%%%%%%%%%%%%%%%%%%%%%%%%%%%%%%%%%%%%%%%%%%%%
\section{The case of $\mathfrak{osp}(m|2n)$}
\label{ospsec}

The main achievement of \cite{MR2369839} was the discovery of special 
tensors, responsible for the uniqueness of the Joseph ideal as a special 
ideal in the universal enveloping algebra of simple complex Lie algebras.
It is then explained in \cite{som} that this tensor is precisely 
responsible for the deformation theory and Poincar\'e–-Birkhoff–-Witt (PBW)
theorem for non-homogeneous quadratic algebras of Koszul type. We 
expect the special tensors constructed in the article will 
play an analogous role in the category of simple complex Lie superalgebras.
In this paper we use these tensors to prove a characterization of the
Joseph ideals, which is a natural analogue of the classical characterization in \cite{Garfinkle}. First derive these results for the positive root system corresponding to $\osp$, in this section.

%%%%%%%%%%%%%%%%%%%%%%%%%%%%%%%%%%%%%%%%%%%%%%%%%%%%%%%%%%%%%%%%%%%%%%%%%%%%%%%%%%%%%%%%%%

\subsection{The Joseph-like ideal for $\mathfrak{osp}(m|2n)$}
\label{subsecosp}
In the following lemma we prove the existence of a specific tensor, which exists for simple Lie algebras by the results in \cite{MR1631302}. We also obtain the explicit form of this tensor, similar to the corresponding result for $\mathfrak{so}(m)$, $\mathfrak{sp}(2n)$ and $\mathfrak{sl}(n)$ in \cite{MR2369839}.
\begin{lemma}
\label{dimhom1}
For $\mathfrak{g}=\osp$, let $\Phi$ denote the composition of $\mg$-module morphisms
\begin{eqnarray*}
\mathfrak{g}\wedge \mathfrak{g}\otimes\mathfrak{g}\hookrightarrow \mathfrak{g}\otimes \mg\otimes \mg \rightarrow \mg\otimes \mg\topa\mg,
\end{eqnarray*}
for $M=m-2n\not\in\{ 1,2\}$. Then
\begin{eqnarray*}
\dim \Hom_\mg(\mg, \ker \Phi)&\ge&1.
\end{eqnarray*}
\end{lemma}
\begin{proof}
We have to construct a tensor $S^{abcdef}$ in $\mathfrak{g}\otimes\mg\otimes \mg$ starting from a tensor  $T\in\mathfrak{g}$ such that the $abcd$ part is in $\mathfrak{g}\wedge\mathfrak{g}$ while the $cdef$ part contains no $\mg\topa\mg$-piece.

We will use the total trace part \eqref{tracepart} in $cdef$
\begin{eqnarray*}
-2g^{de}g^{cf}T^{ab}+2(-1)^{[c][d]}g^{ce}g^{df}T^{ab}
\end{eqnarray*}
which clearly can be regarded as an element of $\mathfrak{g}\otimes\mg\otimes \mg$. By defining, for each $a,b$, $F^{cf}=g^{bc}T^{af}-(-1)^{[a][b]}g^{ac}T^{bf}$ as an element of $\mC^{m|2n}\otimes \mC^{m|2n}$ we can consider the embedding into $\mg\otimes\mg$ of Lemma \ref{embeddingCC},
\begin{eqnarray*}
g^{de}F^{cf}-(-1)^{[d][e]}g^{df}F^{ce}-(-1)^{[c][d]}g^{ce}F^{df}+(-1)^{[c][d]+[e][f]}g^{cf}F^{de},
\end{eqnarray*}
which is in $\mathfrak{g}\otimes\mg\otimes \mg$ as well. By construction, the sum of these tensors
\begin{eqnarray*}
& &-2g^{de}g^{cf}T^{ab}+2(-1)^{[c][d]}g^{ce}g^{df}T^{ab}+g^{de}g^{bc}T^{af}-(-1)^{[a][b]}g^{de}g^{ac}T^{bf}-(-1)^{[d][e]}g^{df}g^{bc}T^{ae}\\
&&+(-1)^{[a][b]+[d][e]}g^{df}g^{ac}T^{be}-(-1)^{[c][d]}g^{ce}g^{bd}T^{af}+(-1)^{[a][b]+[c][d]}g^{ce}g^{ad}T^{bf}\\
&&+(-1)^{[c][d]+[e][f]}g^{cf}g^{bd}T^{ae}-(-1)^{[c][d]+[e][f]+[a][b]}g^{cf}g^{ad}T^{be},
\end{eqnarray*}
which we denote by $U^{abcdef}$, has no Cartan product part in $cdef$. Then we define 
\begin{eqnarray*}
S^{abcdef}&=&U^{abcdef}-(-1)^{([a]+[b])([c]+[d])}U^{cdabef}
\end{eqnarray*}
which by definition is an element of $\mathfrak{g}\wedge\mathfrak{g}\otimes\mathfrak{g}$. It remains to be checked whether the second term of $S$ has no Cartan piece in $cdef$. We therefore write it out as $-(-1)^{([a]+[b])([c]+[d])}U^{cdabef}=$
\begin{eqnarray*}
&&2(-1)^{([a]+[b])([c]+[d])}g^{be}g^{af}T^{cd}-2(-1)^{([a]+[b])([c]+[d])+[a][b]}g^{ae}g^{bf}T^{cd}-(-1)^{([a]+[b])[c]+[a][b]}g^{be}g^{ad}T^{cf}\\
&&+(-1)^{[b]([a]+[d])}g^{be}g^{ac}T^{df}+(-1)^{[a]([c]+[b])+[b]([c]+[e])}g^{bf}g^{ad}T^{ce}-(-1)^{([b]([a]+[d]+[e])}g^{bf}g^{ac}T^{de}\\
&&+(-1)^{[c]([a]+[b])}g^{ae}g^{bd}T^{cf}-(-1)^{[a][d]}g^{bc}g^{ae}T^{df}-(-1)^{[a][e]+([a]+[b])[c]}g^{bd}g^{af}T^{ce}\\
&&+(-1)^{[a]([d]+[e])}g^{bc}g^{af}T^{de}.
\end{eqnarray*}
We take the projection onto $\mathfrak{g}\otimes\mg\odot\mg$. Since everything except the first two terms is already super symmetric in $cd\leftrightarrow ef$, this is given by
\begin{eqnarray*}
&&(-1)^{([a]+[b])([c]+[d])}g^{be}g^{af}T^{cd}+g^{bc}g^{ad}T^{ef}-(-1)^{([a]+[b])([c]+[d])+[a][b]}g^{ae}g^{bf}T^{cd}\\
&&-(-1)^{([a][b]}g^{ac}g^{bd}T^{ef}-(-1)^{([a]+[b])[c]+[a][b]}g^{be}g^{ad}T^{cf}+(-1)^{[b]([a]+[d])}g^{be}g^{ac}T^{df}\\
&&+(-1)^{[a]([c]+[b])+[b]([c]+[e])}g^{bf}g^{ad}T^{ce}-(-1)^{([b]([a]+[d]+[e])}g^{bf}g^{ac}T^{de}+(-1)^{[c]([a]+[b])}g^{ae}g^{bd}T^{cf}\\
&&-(-1)^{[a][d]}g^{bc}g^{ae}T^{df}-(-1)^{[a][e]+([a]+[b])[c]}g^{bd}g^{af}T^{ce}+(-1)^{[a]([d]+[e])}g^{bc}g^{af}T^{de}.
\end{eqnarray*}
It can be checked in a straightforward manner that this is of the form in equation \eqref{FormCartan2}, which is $\cA^{im}=K^{m|2n}_{\epsilon_1+\epsilon_2+\epsilon_3+\epsilon_4}$, see Remark \ref{qsymm}. Therefore $S\in \ker(\Phi)$ for each $T\in\mg$. It remains to be proven that $T\to S$ is a $\mg$-module morphism. Since $U\to S$ is a $\mg$-module morphism we just need to prove that $T\to U$ corresponds to one. The first part of $U$ is trivial, the second one follows immediately from the fact that $T^{ab}\to g^{bc}T^{ad}$, $\mg\to \otimes^4\mC^{m|2n}$ is a $\mg$-module morphism because $T\to U$ is the composition of this with super anti-symmetrization.
\end{proof}

In the following theorem, this tensor will be used to prove that except for one special value of the parameter, the codimension inside $\otimes\mg$  of a 1-parameter family of ideals is finite.

\begin{theorem}
\label{quadJ1}
Consider $\mathfrak{g}=\mathfrak{osp}(m|2n)$ with $M=m-2n$ for $M\not\in\{ 1,2\}$. If $\lambda\not=-\frac{M-4}{4(M-1)}$, the two-sided ideal $\cJ^1_\lambda$ in $\otimes\mathfrak{g}=\oplus_{k=0}^\infty \otimes^k\mg$ generated by
\begin{eqnarray}
\label{quadrel1}
X\otimes Y-X\topa Y-\frac{1}{2}[X,Y]-\lambda\langle X,Y \rangle&\in&\mg\otimes\mg\,\oplus\,\mg\,\oplus\,\mC
\end{eqnarray}
for every $X,Y\in\mg$, contains $\mathfrak{g}$.
\end{theorem}
\begin{proof}
We consider the tensor $S$ from the proof of Lemma \ref{dimhom1} and simplify it inside $\otimes\mg/\cJ^1_\lambda$ in two 
different ways. Since $S^{abcdef}$ has no $\topa$-part in $cdef$, the equality
\begin{eqnarray*}
S^{abcdef}&\cong& \frac{1}{2}\left(S^{abc}{}_d{}^{df}-(-1)^{[c][f]}S^{abf}{}_d{}^{dc}\right)+\lambda S^{ab}{}_{cd}{}^{dc}
\end{eqnarray*}
holds inside $\otimes \mg/\cJ_\lambda^1$. A direct calculation shows
\begin{eqnarray*}
S^{abc}{}_d{}^{df}&=&(4-2M)g^{cf}T^{ab}+(3-M)(-1)^{[a][b]}g^{ac}T^{bf}+(M-3)g^{bc}T^{af}\\
&+&(-1)^{[b][c]}g^{bf}T^{ac}-(-1)^{[a]([b]+[c])}g^{af}T^{bc},
\end{eqnarray*}
which yields the equality
\begin{eqnarray*}
&&S^{abcdef}\cong \frac{M-4}{2}(g^{bc}T^{af}-(-1)^{[a][b]}g^{ac}T^{bf}-(-1)^{[c][f]}g^{bf}T^{ac}\\
&&+(-1)^{[a][b]+[c][f]}g^{af}T^{bc})-2\lambda(M-2)(M-1)T^{ab}.
\end{eqnarray*}
The first term on the right hand side is the embedding of $\mg\cong\mC^{m|2n}\wedge\mC^{m|2n}$ into $\mg\wedge\mg$ of Lemma \ref{embeddingCC}. Inside $\otimes \mg/\cJ_\lambda^1$ such a tensor $Y^{abcf}$ is therefore equivalent to $Y^a{}_b{}^{bf}$. Substituting this finally yields
\begin{eqnarray*}
S^{abcdef}&\cong&(M-2)\left(\frac{M-4}{2}-2\lambda(M-1)\right)T^{ab}.
\end{eqnarray*}

We can also reduce the tensor $S$ by using the defining quadratic relation of $\cJ_\lambda^1$ on $abcd$. The tensor $S$ clearly contains no Cartan part and no $K_0^{m|2n}$-part in $abcd$ since $S\in\mg\wedge\mg\otimes\mg$. Therefore we calculate $S^a{}_{b}{}^{bdef}=$
\begin{eqnarray*}
(M-4)\left(g^{de}T^{af}-(-1)^{[a][d]}g^{ae}T^{df}-(-1)^{[e][f]}g^{df}T^{ae}+(-1)^{[a][d]+[e][f]}g^{af}T^{de}\right).
\end{eqnarray*}
This tensor again corresponds to the embedding $\mg\hookrightarrow \mg\wedge\mg$ of Lemma \ref{embeddingCC}. This immediately implies $S{}_{ab}{}^{baef}=0$ and
\begin{eqnarray*}
S^{abcdef}&\cong&(M-4)(M-2)T^{ab}.
\end{eqnarray*}
Combining the two reductions of $S$ shows that
\begin{eqnarray*}
(M-2)\left(\frac{M-4}{2}+2\lambda(M-1)\right)T^{ab}&\cong&0
\end{eqnarray*}
inside $\otimes \mg /\cJ_\lambda^1$ or $\mg\subset \cJ^1_\lambda$ if $\lambda\not=-\frac{M-4}{4(M-1)}$.
\end{proof}

If $\mg\subset \cJ^1_\lambda$, then $\oplus_{k>0}\otimes^k\mg\subset\cJ_1^\lambda$, which implies that either $\cJ^1_\lambda=\otimes\mg$ or $\cJ^1_\lambda=\oplus_{k>0}\otimes^k\mg$. From equation \eqref{quadrel1} it is clear that when $\oplus_{k>0}\otimes^k\mg\subset\cJ_1^\lambda$, $\mC\subset \cJ^1_\lambda$ if and only if $\lambda\not=0$.

The ideal $\cJ^1_\lambda$ in $\otimes\mg$ can be associated to an ideal $\mathfrak{J}^1_\lambda$ in the universal enveloping algebra $\cU(\mg)$. Notice that the relation \eqref{quadrel1} may be split into the super skew and symmetric part
\begin{eqnarray*}
&&X\otimes Y -(-1)^{|X||Y|}Y\otimes X -[X,Y]\mbox{ and}\\
&&X\otimes Y+ (-1)^{|X||Y|}Y\otimes X -2X\topa Y-2\lambda\langle X,Y\rangle.
\end{eqnarray*}
The algebra $A_\lambda=\otimes \mg/\cJ^1_\lambda$ can therefore be realized in two steps. First we take the quotient with respect to super anti-symmetric part, which yields $\cU(\mg)$. The corresponding image ideal in $\cU(\mg)$ defined by the second relation is then $\mathfrak{J}^1_\lambda$, and $A_\lambda\cong\cU(\mg)/\mathfrak{J}^1_\lambda$.

\begin{theorem}
\label{1cases}
For $\mathfrak{g}=\osp$ with $M=m-2n$ different from $1,2$, the ideal $\mathfrak{J}^1_\lambda$
\begin{itemize}
\item is equal to $\cU(\mathfrak{g})$ if $\lambda\not=-\frac{M-4}{4(M-1)}$ and $\lambda\not=0$
\item is equal to $\cU_+(\mg)=\cU(\mg)\mg$ if $\lambda=0$ and $M\not=4$
\item has an infinite codimension in $\cU(\mathfrak{g})$ if $\lambda=-\frac{M-4}{4(M-1)}$.
\end{itemize}
\end{theorem}
\begin{proof}
The first two properties follow from the considerations above. The third statement follows from the relation of the ideal $\mathfrak{J}^1_\lambda$ with the annihilator ideal of an infinite dimensional representation of $\mg$ in Theorem \ref{idealrep}.
\end{proof}

If $\lambda$ reaches the critical value for which $\cJ_\lambda^1$ (or $\mathfrak{J}_\lambda^1$) has infinite codimension we denote the ideals as $\cJ^1$ (or $\mathfrak{J}^1$) and call $\mathfrak{J}^1$ the Joseph ideal of $\osp$.

\begin{remark}
\label{junction1}
{\rm One can take the junction $\mathfrak{J}^1\cap\cU(\mathfrak{so}(m))$ using the embedding $\mathfrak{so}(m)\hookrightarrow\mathfrak{osp}(m|2n)$. However it can be checked that this ideal is not of the form of the corresponding classical Joseph ideal for $\mathfrak{so}(m)$ in \cite{MR1108044, MR1631302, MR2369839}. This is closely related to Remark \ref{Cartandiff} and the subsequent Remark \ref{remarkrealso}.}
\end{remark}

In \cite{Garfinkle} it was proven that if for an ideal $J\subset\cU(\mathfrak{so}(m))$ with infinite codimension, the corresponding graded ideal $gr(J)$ in $S(\mathfrak{so}(m))= \odot\mathfrak{so}(m)$ satisfies the property $gr(J)\cap S_2(\mathfrak{so}(m))=E$, with $E$ the sum of all irreducible representations inside $\mathfrak{so}(m)\odot\mathfrak{so}(m)$ except the Cartan product, then $J$ is the Joseph ideal. With the obtained results we can now prove the same statement for $\osp$.

\begin{theorem}
\label{Garresult}
Consider a two-sided ideal $\mathfrak{K}$ in $\cU(\mg)$ for $\mg=\osp$ with $m-2n>2$. If $\mathfrak{K}$ has infinite codimension and the associated graded ideal $gr(\mathfrak{K})$ in $S(\mg)=\odot\mg$ satisfies
\begin{eqnarray*}
\left(gr(\mathfrak{K})\,\cap\, \odot^2\mg\right)\,\oplus\, \mg\topa\mg&=&\odot^2\mg,
\end{eqnarray*}
then $\mathfrak{K}$ is equal to the Joseph ideal $\mathfrak{J}^1$.
\end{theorem}
\begin{proof}
Define the ideal $\cK$ in $\otimes\mg$ as the kernel of the composition of the projections
\begin{eqnarray*}
\otimes\mg\to\cU(\mg)\to\cU(\mg)/\mathfrak{K}.
\end{eqnarray*}
We define $\cK_k$, a subspace of $\otimes^k\mg$, as the projection of $\cK\cap\left(\oplus_{j=0}^k\otimes^j \mg\right)$ onto $\otimes^k\mg$. By construction, each $\cK_k$ is a closed subspace under the adjoint action of $\mg$. It is also clear that $\mg\wedge\mg$ is inside $\cK_2$ since for each $X,Y\in\mg$, $2X\wedge Y-[X,Y]$ is inside the kernel of the first projection. Since $\left(gr(\mathfrak{K})\,\cap\, \odot^2\mg\right)$ is naturally embedded in $\cK_2$, we obtain
\begin{eqnarray*}
\cK_2\,\oplus\,\mg\topa\mg&=&\mg\otimes\mg.
\end{eqnarray*}
Therefore $\cK_2$ is equal to the representation $I_2$, defined in Equation \eqref{defBk} for $\mg=\osp$ and $V$ the adjoint representation. Since $\cK$ is also a two-sided ideal, it follows that $\cK_k\supset \cK_{k-1}\otimes\mg+\mg\otimes\cK_{k-1}$, which results in $\cK_k\supset I_k$, by equation \eqref{altdefIk}. Likewise, for the Joseph ideal $\cJ^1$ in $\otimes\mg$, we can define $\left(\cJ^1\right)_k$, which is equal to $I_k$.

From the assumed property of $gr(\mathfrak{K})$ it follows that for each $X,Y\in\mg$, there must be at least one element of the form $XY+ (-1)^{|X||Y|}YX -2X\topa Y+Z(X,Y)+c(X,Y)$ inside $\mathfrak{K}\cap\cU_2(\mg)$ with $Z(X,Y)\in\mg$ and $c(X,Y)\in\mC$. Since $\mathfrak{K}$ is a two-sided ideal it follows that $Z$ and $c$ extend to a $\mg$-module morphism $\mg\odot\mg\to\mg$ and $\mg\odot\mg\to\mC$, which by Theorem \ref{decomposition} imply that $Z=0$ and $c(X,Y)=\lambda \langle X,Y\rangle$ for some constant $\lambda$. Theorem \ref{quadJ1} then implies that the only possible value of $\lambda$ which does not contradict the infinite codimension of $\mathfrak{K}$ is $\lambda=-(M-4)/(4(M-1))$. From this it follows immediately that $\cK$ contains $\cJ^1$ and $\cK\cap\left(\oplus_{j=0}^{2}\otimes^j \mg\right)=\cJ^1\cap\left(\oplus_{j=0}^{2}\otimes^j \mg\right)$.

Now if $\cK$ were bigger than $\cJ^1$, then for one value of $k$, $\cK\cap\left(\oplus_{j=0}^k\otimes^j \mg\right)$ would be bigger than $\cJ^1\cap\left(\oplus_{j=0}^k\otimes^j \mg\right)$ while $\cK\cap\left(\oplus_{j=0}^{k-1}\otimes^j \mg\right)=\cJ^1\cap\left(\oplus_{j=0}^{k-1}\otimes^j \mg\right)$ holds, since $\cK\cap\left(\oplus_{j=0}^{2}\otimes^j \mg\right)=\cJ^1\cap\left(\oplus_{j=0}^{2}\otimes^j \mg\right)$. This implies that $\cK_k$ is bigger than $I_k$, otherwise every element in $\cK\cap\left(\oplus_{j=0}^k\otimes^j \mg\right)$ has a corresponding element in $\cJ^1\cap\left(\oplus_{j=0}^k\otimes^j \mg\right)$ with the same leading term, so by subtracting these two we obtain elements in $\cK\cap\left(\oplus_{j=0}^{k-1}\otimes^j \mg\right)$ which are not in $\cJ^1\cap\left(\oplus_{j=0}^{k-1}\otimes^j \mg\right)$.

If there would be such a $\cK_k$ which is strictly bigger than $I_k$, than by Theorem \ref{Lemequiv} and Lemma \ref{Cartanosp} (1)$\leftrightarrow$(3) it would follow that $\cK_k=\otimes^k\mg$, which implies $\cK_j=\otimes^j\mg$ for all $j\ge k$. Since $\cK$ has infinite codimension this is not possible and the theorem is proven.
\end{proof}

%%%%%%%%%%%%%%%%%%%%%%%%%%%%%%%%%%%%%%%%%%%%%%%%%%%%%%%%%%%%%%%%%%%%%%%

\subsection{The corresponding representation of $\osp$}
\label{secreposp}
In this section we will consider the superspace $\mR^{p|2n}$ from Section \ref{preli}, for $p=m$ and $p=m-2$. This corresponds to a super version of the ambient space method, used for the minimal representation and the Joseph ideal for $\mathfrak{so}(m)$ in e.g. \cite{MR1108044, MR2180410, Higherpowers, MR2020550}.

 For $p=m-2$ we will use the notations from Section \ref{preli}, i.e. the operators $\Delta$, $R^2$ and $\mE$ of equation \eqref{DRE} and capital letters for the variables $\{X_i|\, i=1,\cdots,m-2+2n\}$.

For $p=m$ we use small letters for the variables, $\{x_a|\, a=1,\cdots,m+2n\}$ where the first $m$ are commuting and the last $2n$ anti-commuting. The corresponding $\mathfrak{sl}(2)$ realization on $\mR^{m|2n}$ will be denoted by $\widetilde{\Delta}$, $\widetilde{R}^2$, $\widetilde{\mE}$. These correspond to the definition in equation \eqref{DRE} with substitution $h\to g$. For the partial differential operators we use the notation
\begin{eqnarray*}
\partial_a=\sum_b g_{ba}\partial_{x_b}&\quad\mbox{which implies}\quad& \partial_a (x_c)=g_{ca},
\end{eqnarray*}
for $g$ the orthosymplectic metric on $\mR^{m|2n}$ used in the definition of $\mathfrak{osp}(m|2n)$.

We consider the canonical realization of $\osp$ as first order differential operators on $C^\infty(\mR^{m|2n})$, 
\[\mD_V=\sum_{a,b}V^{ab}x_a\partial_b=\sum_{a,c}V^a{}_c\, x_a\partial_{x_c},\]
for $V\in\osp$, see e.g. \cite{OSpHarm}. Since this corresponds to a representation of $\osp$, we know that
\begin{eqnarray}
\label{relantis}
\mD_U\mD_V-(-1)^{|U||V|}\mD_V\mD_U&=&\mD_{[U,V]}.
\end{eqnarray}
This action of $\osp$ commutes with the generators of $\mathfrak{sl}(2)$ given by $\widetilde{\Delta}$, $\widetilde{R}^2$ and $\widetilde{\mE}$. We assume the metric $h\in\mR^{(m-2+2n)\times (m-2+2n)}$ on $\mR^{m-2|2n}$ and the metric $g\in\mR^{(m+2n)\times (m+2n)}$ on the bigger (usualy termed ambient) 
space $\mR^{m|2n}$ to be related by
\begin{eqnarray*}
g&=&\left( \begin{array}{ccc}0&1&0\\  \vspace{-3.5mm} \\1&0&0\\  \vspace{-3.5mm} \\ 0&0&h \end{array} \right).
\end{eqnarray*}
With the association $x_{i+2}=X_{i}$ for $i=1,\cdots,m-2+2n$, this implies
\begin{eqnarray*}
\widetilde{\Delta}=2\partial_{x_1}\partial_{x_2}+\Delta,\quad \widetilde{R}^2=2x_1x_2+R^2\quad\mbox{and }\quad \widetilde{\mE}=x_1\partial_{x_1}+x_2\partial_{x_2}+\mE.
\end{eqnarray*}
The function space $C^\infty(\mR^{m|2n}_+):=C^\infty(\mR^m_+)\otimes\Lambda_{2n}$ is the algebra of smooth functions on the half space $\mR^m_+=\{(x_1,\cdots,x_m)|\, x_1>0\}$ with values in the Grassmann algebra $\Lambda_{2n}$. We take the quotient space $C^\infty(\mR^{m|2n}_+)/(\widetilde{R}^2)$ with respect to the ideal generated by the function $\widetilde{R}^2$. Then we can restrict to functions of homogeneous degree $\alpha$, for $\alpha\in\mR$:
\begin{eqnarray*}
\cF_{\alpha}&=&\{f\in C^\infty(\mR^{m|2n}_+)/(\widetilde{R}^2)|\, \widetilde{\mE}f=\alpha f\}\subset  C^\infty(\mR^{m|2n}_+)/(\widetilde{R}^2).
\end{eqnarray*}
The space $\cF_{2-M/2}$ has a useful property. The Laplace operator $\widetilde{\Delta}$ maps functions in $C^\infty(\mR^{m|2n}_+)$ of homogeneous degree $2-M/2$ to functions of degree $-M/2$. Consider $\widetilde{R}^2k\in C^\infty(\mR^{m|2n}_+)$ of degree $2-M/2$ (so $k$ is of degree $-M/2$), the equation
\begin{eqnarray*}
\widetilde{\Delta}\widetilde{R}^2k&=&\widetilde{R}^2\widetilde{\Delta}k+(4\widetilde{\mE}+2M)k=\widetilde{R}^2\widetilde{\Delta}k
\end{eqnarray*}
is an immediate consequence of the $\mathfrak{sl}(2)$-relations among $\widetilde{\Delta}/2$, $\widetilde{R}^2/2$ and $\widetilde{\mE}+M/2$. This implies that $\widetilde{\Delta}$ acting from $C^\infty(\mR^{m|2n}_+)_{2-M/2}$ to $C^\infty(\mR^{m|2n}_+)_{-M/2}$ naturally descends to an action from $\cF_{2-M/2}$ to $\cF_{-M/2}$. Therefore we can define a subspace
\begin{eqnarray*}
\widetilde{\cH}&=&\{f\in\cF_{2-M/2}\,|\,\widetilde{\Delta} f=0\}\subset \cF_{2-M/2}.
\end{eqnarray*}
Since the action of $\osp$ commutes with $\widetilde{\Delta}$, $\widetilde{R}^2$ and $\widetilde{\mE}$, the space $\widetilde{\cH}$ is an $\osp$-module. To consider the ideal in the universal enveloping algebra corresponding to this representation we need to consider the composition of two vector fields $\mD_U$ and $\mD_V$ for $U,V\in\osp$. This yields
\begin{eqnarray*}
\mD_U\mD_V&=&(-1)^{[b][c]}U^{ab}V^{cd}x_ax_c\partial_b\partial_d+U^{a}{}_bV^{bd}x_a\partial_d.
\end{eqnarray*}
Therefore, the $\mg$-module morphism $\otimes\mg\to$ Diff$(\mR^{m|2n})$, given by 
\[V_1\otimes V_2\otimes \cdots \otimes V_k\to\mD_{V_1\otimes V_2\otimes \cdots V_k}\equiv\mD_{V_1}\mD_{V_2}\cdots\mD_{V_k},\]
satisfies
\begin{eqnarray}
\label{DXgg}
\mD_X&=&(-1)^{[b][c]}X^{abcd}x_ax_c\partial_b\partial_d+X^{a}{}_b{}^{bd}x_a\partial_d,
\end{eqnarray}
for $X\in \mg\otimes\mg$.

\begin{lemma}
\label{DUDV}
When acting on $\widetilde{\cH}$, the composition of vector fields satisfies
\begin{eqnarray*}
\mD_{U\otimes V}=\mD_U\mD_V&\cong & \mD_{U\topa V}+\frac{1}{2}\mD_{[U,V]}+\lambda\mD_{\langle U,V\rangle}
\end{eqnarray*}
for $\lambda=-\frac{M-4}{4(M-1)}$ and for all $U,V\in \mg=\osp$ with $m-2n\not\in\{1,2\}$. Moreover, when acting on $\cF_{2-M/2}$ the composition satisfies
\begin{eqnarray*}
\mD_{U\otimes V}&\equiv & \mD_{U\topa V}+\frac{1}{2}\mD_{[U,V]}+\lambda\mD_{\langle U,V\rangle}\quad\mod\,\widetilde{\Delta}.
\end{eqnarray*}
\end{lemma}
\begin{proof}
The proof of this relation can be split up into the super symmetric and the super anti-symmetric part. When evaluated on $\widetilde{\cH}$, we have to get
\begin{eqnarray*}
\mD_U\mD_V+(-1)^{|U||V|}\mD_V\mD_U&\cong & 2\mD_{U\topa V}+2\lambda\mD_{\langle U,V\rangle}\quad\mbox{and}\\
\mD_U\mD_V-(-1)^{|U||V|}\mD_V\mD_U&\cong &\mD_{[U,V]}.
\end{eqnarray*}
The second equation trivially holds, see equation \eqref{relantis}. The left-hand side of the first equation represents $\mD_X$ for a general tensor $X\in\mg\odot\mg$. First we exclude the case $M=m-2n=0$. By choosing $X$ to be in one of the four submodules of $\mg\odot\mg$ in Theorem \ref{decomposition}, the statement is equivalent to 
\begin{eqnarray*}
& & \mD_X=0\qquad\quad \mbox{ on }\quad \widetilde{\cH}\quad \mbox{ if } X\in K^{m|2n}_{\epsilon_1+\epsilon_2+\epsilon_3+\epsilon_4},\\ 
& & \mD_X=0\qquad\quad\mbox{ on }\quad\widetilde{\cH}\quad \mbox{ if } X\in K^{m|2n}_{2\epsilon_1},\\
& & \mD_X=-\frac{M-4}{4(M-1)}\sum_{ab}X_{ab}{}^{ba}\quad\mbox{ on }\widetilde{\cH}\mbox{ if }X\in K^{m|2n}_0.
\end{eqnarray*}
First consider $X\in K^{m|2n}_{2\epsilon_1}$, so
\begin{eqnarray*}
X^{abcd}&=&\frac{1}{4}\left(g^{bc}A^{ad}-(-1)^{[a][b]}g^{ac}A^{bd}-(-1)^{[c][d]}g^{bd}A^{ac}+(-1)^{[a][b]+[c][d]}g^{ad}A^{bc}\right),
\end{eqnarray*}
with $A\in \odot^2_0\mC^{m|2n}$ a super symmetric traceless tensor, see Lemma \ref{embeddingCC}. Equation \eqref{DXgg} then implies
\begin{eqnarray*}
4\mD_X&=&A^{ad}x_a\widetilde{\mE}\partial_d+MA^{ad}x_a\partial_d-A^{bd}\widetilde{R}^2\partial_b\partial_d-A^{ad}x_a\partial_d\\
&-&A^{ac}x_ax_c\widetilde{\Delta}-A^{ad}x_a\partial_d+(-1)^{[b][c]}A^{bc}x_c\widetilde{\mE}\partial_b.
\end{eqnarray*}
When evaluated on the space $\widetilde{\cH}$, this reduces to
\begin{eqnarray*}
4\mD_X&=&(1-\frac{M}{2})A^{ad}x_a\partial_d+MA^{ad}x_a\partial_d-A^{ad}x_a\partial_d-A^{ad}x_a\partial_d+(1-\frac{M}{2})(-1)^{[a][d]}A^{da}x_a\partial_d\\
&=&(\frac{M}{2}-1)\left(A^{ad}-(-1)^{[a][d]}A^{da}\right)x_a\partial_d,
\end{eqnarray*}
which is zero by the super symmetry of $A$. 

As for the total trace part $K_0^{m|2n}$ in \eqref{tracepart}, $X^{abcd}=\frac{1}{2}\left(g^{bc}g^{ad}-(-1)^{[a][b]}g^{ac}g^{bd}\right)$, the associated differential operator is
\begin{eqnarray*}
\mD_X&=&\frac{1}{2}\left(\widetilde{\mE}(\widetilde{\mE}-1)-\widetilde{R}^2\widetilde{\Delta}+(M-1)\widetilde{\mE}\right).
\end{eqnarray*}
When evaluated on the elements of $\widetilde{\cH}$ this becomes
\begin{eqnarray*}
\mD_X&\cong &-\frac{M}{8}(M-4).
\end{eqnarray*}
The comparison with $X_{ab}{}^{ba}=M(M-1)/2$, see e.g. equation \eqref{tracepphi}, then yields the result.

The result for $K_{\epsilon_1+\epsilon_2+\epsilon_3+\epsilon_4}^{m|2n}$ follows from the explicit form of the projection operator in equation \eqref{FormCartan2},
\begin{eqnarray*}
X^{abcd}&=&\frac{1}{3}\left(V^{abcd}+(-1)^{[a]([b]+[c])}V^{bcad}-(-1)^{[b][c]}V^{acbd}\right)
\end{eqnarray*}
with $V^{abcd}\in \mg\odot\mg$. This is totally traceless, therefore
\begin{eqnarray*}
3\mD_X&=&(-1)^{[b][c]}V^{abcd}x_ax_c\partial_b\partial_d+(-1)^{[a]([b]+[c])+[b][c]}V^{bcad}x_ax_c\partial_b\partial_d-V^{acbd}x_ax_c\partial_b\partial_d.
\end{eqnarray*}
This is zero since $V^{acbd}=-(-1)^{[a][c]}V^{cabd}$ while $x_ax_c=(-1)^{[a][c]}x_cx_a$.

For the case $M=0$ this can be proven very similarly. Instead of considering $K^{m|2n}_{2\epsilon_1}$ and $K^{m|2n}_{0}$ independently one needs to prove that
\begin{eqnarray*}
D_{X}&=&-\sum_{ab}X_{ab}{}^{ba}
\end{eqnarray*}
for $X=\phi(A)$ with $\phi$ defined in the proof of Lemma \ref{embeddingCC} and $A\in\mC^{m|2n}\odot\mC^{m|2n}$, which follows from a direct calculation.

The slightly stronger second statement follows immediately from the calculations above.
\end{proof}

Note that the right-hand side in Lemma \ref{DUDV} cannot be further simplified, i.e. all terms are non-zero and do 
not correspond to lower-order differential operators. This follows from the fact that they correspond to irreducible $\osp$-representations.

\begin{lemma}
\label{alphabeta}
The function space $\cF_{2-M/2}\subset C^\infty(\mR_+^{m|2n})/(\widetilde{R}^2)$ is isomorphic to $C^\infty(\mR^{m-2|2n})$.
\end{lemma}
\begin{proof}
We consider the super vector space morphisms
\begin{eqnarray*}
& & \alpha:\cF_{2-M/2}\to C^\infty(\mR^{m-2|2n}),\\
& & \alpha(f)(\mathbf{x})=f(1,-R^2/2,\mathbf{x}),\\
& & \beta:C^\infty(\mR^{m-2|2n})\to\cF_{2-M/2},\\
& & \beta(F)(x_1,x_2,\mathbf{x})=x_1^{2-\frac{M}{2}}F(\mathbf{x}/{x_1}),
\end{eqnarray*}
with $\bold{x}=(X_1,\cdots,X_{m-2+2n})=(x_3,\cdots,x_{m+2n})$. Functions in $R^2$ are defined by a finite Taylor series in the Grassmann variables. The morphism $\alpha$ is well-defined since $\alpha(\widetilde{R}^2h)=0$.
The composition $\beta\circ\alpha$ on $\cF_{2-M/2}$ satisfies
\begin{eqnarray*}
\beta\circ\alpha (f) (x_1,x_2,\mathbf{x})&=&x_1^{2-\frac{M}{2}}f(1,-R^2/(2x_1^2),\mathbf{x}/x_1)\\
&=&f(x_1,-R^2/(2x_1),\mathbf{x})=f(x_1,x_2,\mathbf{x}).
\end{eqnarray*}
Similarly we can show that $\alpha\circ\beta$ is the identity on $C^\infty(\mR^{m-2|2n})$.
\end{proof}
This isomorphism gives an induced action of $\osp$ on functions on $\mR^{m-2|2n}$, defined by 
\begin{eqnarray}
\label{cDX}
\cD_X&=&\alpha\circ\mD_X\circ\beta \qquad\mbox{for}\quad X\in\osp.
\end{eqnarray}
This extends immediately to $\cD_X$ for $X\in\otimes\mg$.

It is clear that $\widetilde{\Delta}\beta(F)=0$ if and only if $\Delta F=0$, which implies that under the isomorphism between $\cF_{2-M/2}$ and $C^\infty(\mR^{m-2|2n})$, the space $\widetilde{\cH}$ corresponds to the harmonic functions on $\mR^{m-2|2n}$. This gives the harmonic polynomials on $\mR^{m-2|2n}$ and $\osp$-module structure. This is expressed explicitly in the following theorem.
\begin{theorem}
\label{calccDX}
The harmonic polynomials on $\mR^{m-2|2n}$, $\cH=\bigoplus_{k=0}^{\infty}\cH_k$, form a representation of $\osp$, with action given by
\begin{enumerate}
\item
$
\cD_A= \sum_{i,j=1}^{m-2+2n}A^{ij}X_i\partial_{X^j}\,\mbox{for}\, A\in\mathfrak{osp}(m-2|2n)\hookrightarrow \osp ,
$
\item
$\partial_{X^j}$, $\quad X_j \left(2\mE+M-4\right)-R^2\partial_{X^j}$ for $j=1,\cdots,m-2+2n$,
\item $2\mE+M-4$.
\end{enumerate}
\end{theorem}
\begin{proof}
This can be calculated directly from the equality $\cD_X=\alpha\circ\mD_X\circ\beta$ for $X\in\osp$. The first kind corresponds to $X=\left( \begin{array}{cc}0&0\\  \vspace{-3.5mm} \\0&A\\  \end{array} \right)\in\osp$ for $A\in\mathfrak{osp}(m-2|2n)$. The second and third kind correspond respectively to $x_1\partial_{j+2}-x_{j+2}\partial_1$ and $2(x_2\partial_{j+2}-x_{j+2}\partial_2)$. The last one corresponds to $2(x_2\partial_1-x_1\partial_2)$.
\end{proof}

\begin{remark}
\label{remarkrealso}
{\rm It is clear that the realization of $\mathfrak{so}(m)\hookrightarrow\mathfrak{osp}(m|2n)$ given in Theorem \ref{calccDX} does not correspond to the classical realization of $\mathfrak{so}(m)$ as differential operators on $\mR^{m-2}$ preserving the kernel of the Laplace operator. This is the essential difference between the two Joseph-type ideals and their corresponding representations in the current paper. Because for the metaplectic representation of $\mathfrak{spo}(2n|m)$ in Section \ref{sposec}, the restriction to $\mathfrak{sp}(2n)$ gives the ordinary metaplectic representation. This is closely related to Remarks \ref{Cartandiff} and \ref{junction1}.}
\end{remark}

\begin{corollary}
\label{cDcD}
The differential operators $\cD_X$ on $C^\infty(\mR^{m-2|2n})$ for $X\in\osp$ defined in equation 
\eqref{cDX} and calculated explicitly in Theorem \ref{calccDX}, satisfy the relation
\begin{eqnarray*}
\cD_X\cD_Y&\equiv & \cD_{X\topa Y}+\frac{1}{2}\cD_{[X,Y]}+\lambda\cD_{\langle X,Y\rangle}\quad \mod \Delta
\end{eqnarray*}
for $\lambda=-\frac{M-4}{4(M-1)}$, with $X,Y\in\mathfrak{osp}(m|2n)$.
\end{corollary}
\begin{proof}
This is an immediate consequence of Lemma \ref{DUDV} and equation \eqref{cDX}.
\end{proof}

For the next theorem we will make the substitution $m=p+2$ in order to connect more easily with earlier results on harmonic analysis on superspace, as in \cite{OSpHarm, MR2344451}.
\begin{theorem}
Consider the $\mathfrak{osp}(p+2|2n)$-representation $\cH=\bigoplus_{k=0}^\infty \cH_k$ of harmonic polynomials on $\mR^{p|2n}$ from Theorem \ref{calccDX}. The space is an irreducible $\mathfrak{osp}(p+2|2n)$-representation if $p-2n\not\in 2-2\mN$. If $p-2n=2-2q$, the representation $\cH$ is indecomposable but has exactly one (irreducible) subrepresentation, which decomposes as an $\mathfrak{osp}(p|2n)$-representation into irreducible pieces as
\begin{eqnarray*}
\bigoplus_{j=0}^q \cH_j \oplus \bigoplus_{j=q+1}^{2q}R^{2j-2q}\cH_{2q-j}.
\end{eqnarray*}
This subrepresentation is isomorphic to $K^{p+2|2n}_{q\epsilon_1}=K^{p+2|2n}_{(1+n-p/2)\epsilon_1}$.
\end{theorem}
\begin{proof}
In this proof we choose the metric $g$ such that the equalities $X^1=X_2$ and $X^2=X_1$ hold.

First we consider the case $p-2n\not\in-2\mN$. Each space $\cH_k$ is an irreducible $\mathfrak{osp}(p|2n)$-module with highest weight vector $X_1^k$, see Theorem \ref{sphHarm}. It is clear that $\partial_{X_1}$ from Theorem \ref{calccDX} maps $X^k_1$ to $kX_1^{k-1}$ while 
\begin{eqnarray*}
\left(X_1\left(2\mE+p-2n-2\right)-R^2\partial_{X^1}\right)X_1^k&=&(2k+p-2n-2)X_1^{k+1}.
\end{eqnarray*}
So it is clear that if $p-2n\not\in2-2\mN$, $\cH$ is irreducible. If $p-2n=2$, the scalars form a submodule, which does not have a complement representation, while the quotient $\cH/\mC$ is an irreducible $\mathfrak{osp}(p+2|2n)$-representation.

Now we consider the case $p-2n=2-2q$ with $q\in\mN_+$. Theorem \ref{sphHarm} implies that $\cH_k$ is an irreducible $\mathfrak{osp}(p|2n)$-representation if $k\not\in[q+1,2q]$ with highest weight vector $X_1^k$, while if $k\in[q+1,2q]$, $\cH_k$ is an indecomposable highest weight module with highest weight vector $X_1^k$ and with one submodule $R^{2k-2q}\cH_{2q-k}$. 

The action of the elements $\partial_{X_1}$ and $X_1 \left(2\mE-2q\right)-R^2\partial_{X^1}$ shows that each vector $X_1^k$ with $k>q$ generates the entire representation $\cH$, so $\cH$ is an indecomposable $\osp$-representation. It also shows that an $\mathfrak{osp}(p+2|2n)$-submodule cannot contain $X_1^k$ for $k>q$. Since all partial derivatives $\partial_{X_j}$ are elements of $\mathfrak{osp}(p+2|2n)$ each submodule $\cU$ must contain the scalars. The elements $\quad X_1 \left(2\mE-2q\right)-R^2\partial_{X^1}$ then imply that $X_1^k\in\cU$ for $k\le q$. Then we can take the action of the element $\quad X_2 \left(2\mE-2q\right)-R^2\partial_{X^2}$ on $X_1^q$, which shows that also $R^{2j-2q}X_1^{2q-j}$ is inside $\cU$ for $j=q+1,2q$.

Summarizing, this yields that any possible submodule must be of the form
\begin{eqnarray*}
\cU&=&\bigoplus_{j=0}^q \cH_j \oplus \bigoplus_{j=q+1}^{2q}R^{2j-2q}\cH_{2q-j},
\end{eqnarray*}
as an $\mathfrak{osp}(p|2n)$-representation. To prove that this is an $\mathfrak{osp}(p+2|2n)$-representation we only need to show that it is preserved by the action of the operators $\partial_{X^j}$ and $\left(X_l(2\mE-2q)-R^2\partial_{X^l}\right)$. This corresponds to noting that
\begin{eqnarray*}
\left(X_l(2\mE-2q)-R^2\partial_{X^l}\right)R^{2j-2-2q}\cH_{2q-j+1}&\subset& R^{2j-2q}\cH_{2q-j}
\end{eqnarray*}
for $j=q+1,\cdots, 2q$ and $\left(X_l(2\mE-2q)-R^2\partial_{X^l}R^{2j-2-2q}\right)R^{2q}=0$ holds. This follows easily since the right-hand side is always an element of $R^2\cP$ and $\cH_k\cap R^2\cP=R^{2k-2q}\cH_{2q-k}$ for $q+1\le k\le 2q$ and zero otherwise, which is an immediate consequence of Theorem \ref{sphHarm}. The corresponding claim for $\partial_{X^j}$ follows similarly.
\end{proof}
Note that the identification of the submodule with $K^{p+2|2n}_{(1+n-p/2)\epsilon_1}$ yields a branching rule for $\mathfrak{osp}(p|2n)\hookrightarrow\mathfrak{osp}(p+2|2n)$, which follows also from applying Theorem 10 in \cite{OSpHarm} twice:
\begin{eqnarray*}
K^{p+2|2n}_{(1+n-p/2)\epsilon_1}&\cong&K^{p+1|2n}_{(1+n-p/2)\epsilon_1}\,\oplus\, K^{p+1|2n}_{(n-p/2)\epsilon_1}\\
&\cong&\bigoplus_{l=0}^{1+n-p/2}K^{p|2n}_{l\epsilon_1}\,\oplus\,\bigoplus_{j=0}^{n-p/2} K^{p|2n}_{j\epsilon_1}.
\end{eqnarray*}

This representation of $\mathfrak{osp}(m|2n)$ (with $m=p+2$) is not unitarizable, contrary to the classical case. This follows from the fact that, due to the structure of the roots, a faithful unitarizable representation of $\mathfrak{osp}(m|2n)$ remains unitarizable as an $\mathfrak{osp}(m-2|2n)$-representation. The module $\cH$ can not be unitarizable for $\mathfrak{osp}(m-2|2n)$ since it decomposes into finite dimensional representations, which are never unitarizable for orthosymplectic Lie superalgebras different from $\mathfrak{osp}(2|2n)$, see \cite{MR0424886}. If $m-2n\in 4-2\mN$, the non-unitarizability follows immediately from the non-completely reducibility. This non-unitarizability is also related to the fact that the representation corresponds to a representation of the real Lie superalgebra $\mathfrak{osp}(p+1,1|2n;\mR)$, which has no unitary representations, see Theorem 6.3.1 in \cite{Salmasian}.

Now we can prove that the Joseph ideal from Subsection \ref{subsecosp} is related to the annihilator ideal of this representation on the kernel of the super Laplace operator.

\begin{theorem}
\label{idealrep}
If  for $\mg=\osp$, $M=m-2n\not\in\{1,2\}$ holds, the annihilator ideal in the universal enveloping algebra $\cU(\mathfrak{g})$ of the representation of $\mathfrak{osp}(m|2n)$ on harmonic polynomials on $\mR^{m-2|2n}$, or its irreducible quotient space contains the ideal $\mathfrak{J}^1$. In particular this implies that $\cU(\mg)/\mathfrak{J}^1$ is infinite dimensional. If $m-2n>2$, the annihilator ideal is identical to $\mathfrak{J}^1$.
\end{theorem}
\begin{proof}
Corollary \ref{cDcD} implies that the ideal in the universal enveloping algebra of $\osp$, corresponding to the representation $\cH$, contains $\mathfrak{J}^1$. The fact that the annihilator ideal is not bigger than $\mathfrak{J}^1$ in case $m-2n>2$, then follows from Theorem \ref{Garresult}.
\end{proof}

\begin{remark}
{\rm In case equation \eqref{herformCartan} could be proven for $\mg=\osp$ and $V$ the adjoint representation for $m-2n<0$, the equality of the annihilator ideal and Joseph-like ideal would follow from Lemma \ref{Lemequiv} for those cases as well.}
\end{remark}

%%%%%%%%%%%%%%%%%%%%%%%%%%%%%%%%%%%%%%%%%%%%%%%%%%%%%%%%%%%%%%%%%%%%%%%%%%%%%
\section{The case of $\mathfrak{spo}(2n|m)$}
\label{sposec}

In this section we use the second notion of the Cartan product in the second tensor power of $\mg=\mathfrak{spo}(2n|m)$ to construct a second 1-parameter family of ideals in $\cU(\mg)$. Again, only for one value of the parameter the ideal has infinite codimension. In this case the ideal generalizes the Joseph ideal of $\mathfrak{sp}(2n)$. We also show that this ideal is the annihilator ideal of a generalization of the minimal representation of $\mathfrak{sp}(2n)$ to $\mathfrak{spo}(2n|m)$ studied in e.g. \cite{Tensor}.

\subsection{The Joseph-like ideal for $\mathfrak{spo}(2n|m)$}

\begin{lemma}
\label{dimhom2}
For $\mathfrak{g}=\mathfrak{spo}(2n|m)$, let $\Phi$ denote the composition
\begin{eqnarray*}
\mathfrak{g}\wedge \mathfrak{g}\otimes\mathfrak{g}\hookrightarrow \mathfrak{g}\otimes \mg\otimes \mg \rightarrow \mg\otimes \mg\topb\mg.
\end{eqnarray*}
Then,
\begin{eqnarray*}
\dim \Hom_\mg(\mg, \ker \Phi)&\ge&1.
\end{eqnarray*}
\end{lemma}
\begin{proof}
First we construct, from $T\in\mg$, the tensor $U^{abcdef}$
\begin{eqnarray*}
&=&-4g^{de}g^{cf}T^{ab}+4(-1)^{[c][d]}g^{ce}g^{df}T^{ab}-g^{cb}g^{de}T^{af}+(-1)^{[c][d]}g^{db}g^{ce}T^{af}+(-1)^{[e][f]}g^{cb}g^{df}T^{ae}\\
&&-(-1)^{[c][d]+[e][f]}g^{db}g^{cf}T^{ae}+(-1)^{[a][b]}g^{ca}g^{de}T^{bf}-(-1)^{[a][b]+[c][d]}g^{da}g^{ce}T^{bf}\\
&&-(-1)^{[a][b]+[e][f]}g^{ca}g^{df}T^{be}+(-1)^{[a][b]+[c][d]+[e][f]}g^{da}g^{cf}T^{be}
\end{eqnarray*}
of which the $cdef$-part is inside $K_0^{m|2n}\oplus K^{m|2n}_{2\epsilon_1}\oplus K^{m|2n}_{\epsilon_1+\epsilon_2}$. Then we define the tensor 
\[S^{abcdef}=U^{abcdef}-(-1)^{([a]+[b])([c]+[d])}U^{cdabef},\]
the second term is given explicitly by $-(-1)^{([a]+[b])([c]+[d])}U^{cdabef}$
\begin{eqnarray*}
&=&4(-1)^{([a]+[b])([c]+[d])}g^{be}g^{af}T^{cd}-4(-1)^{([a]+[b])([c]+[d])+[a][b]}g^{ae}g^{bf}T^{cd}\\
&&+(-1)^{[b]([c]+[d])+[a][c]}g^{da}g^{be}T^{cf}-(-1)^{([a]+[b])[c]}g^{db}g^{ae}T^{cf}-(-1)^{[b]([c]+[d])+[a][c]+[e][f]}g^{da}g^{bf}T^{ce}\\
&&+(-1)^{[a][c]+[b][c]+[e][f]}g^{db}g^{af}T^{ce}-(-1)^{[b]([c]+[d])}g^{ca}g^{be}T^{df}+(-1)^{[b][d]}g^{cb}g^{ae}T^{df}\\
&&+(-1)^{[b]([c]+[d])+[e][f]}g^{ca}g^{bf}T^{de}-(-1)^{[a][d]+[e][f]}g^{cb}g^{af}T^{de}.
\end{eqnarray*}
It can then be checked that the $cdef$-part of the tensor, after being super symmetrized in $cd\leftrightarrow ef$, is of the form
\begin{eqnarray*}
\frac{2}{3}V^{cdef}+\frac{1}{3}(-1)^{[d][e]}V^{cedf}-\frac{1}{3}(-1)^{[c]([d]+[e])}V^{decf},
\end{eqnarray*}
for $V\in\mg\odot\mg$, which according to Remark \ref{qsymm} corresponds to the projection of $\mg\odot\mg$ onto everything except $\mg\topb\mg$. This shows that $S$ is of the required form.
\end{proof}

This tensor can now be used to prove the following theorem. Since the proof is, like the proof of Theorem \ref{quadJ1}, a generalization of methods in \cite{MR2369839} using the results obtained in the current paper in Section \ref{sectenpow}, we do not give it explicitly.

\begin{theorem}
\label{quadJ2}
Consider $\mathfrak{g}=\mathfrak{spo}(2n|m)$. If $\mu\not=\frac{1}{4}$, the two-sided ideal $\cJ^2_\mu$ in $\otimes\mathfrak{g}$ generated by
\begin{eqnarray*}
X\otimes Y-X\topb Y-\frac{1}{2}[X,Y]-\mu\langle X,Y \rangle
\end{eqnarray*}
for all $X,Y\in\mg$, contains $\mathfrak{g}\subset\otimes \mathfrak{g}$.
\end{theorem}
\begin{proof}
This is proven similarly to Theorem $3.1$ in \cite{MR2369839} by reducing the tensor $S$ from Lemma \ref{dimhom2} in two different ways.
\end{proof}

Again we can consider the associated ideal in $\cU(\mg)$, $\mathfrak{J}^2_\mu$, which is a generalization of the Joseph ideal for $\mathfrak{sp}(2n)$ for the critical value.

\begin{theorem}
\label{cases2}
For $\mg=\mathfrak{spo}(2n|m)$, the ideal $\mathfrak{J}^2_\mu$
\begin{itemize}
\item is equal to $\cU(\mathfrak{g})$ if $\mu\not=1/4$ and $\mu\not=0$
\item is equal to $\cU_+(\mathfrak{g})=\mg\cU(\mg)$ if $\mu=0$
\item has infinite codimension if $\mu=1/4$.
\end{itemize}
\end{theorem}
\begin{proof}
The first two cases follow similarly as in Theorem \ref{1cases}, the last case is stated in the subsequent Theorem \ref{JosephKostant}.
\end{proof}
For the critical value $\mu=1/4$ the ideals will be denoted by $\cJ^2$ and $\mathfrak{J}^2$.

\begin{remark}
{\rm Contrary to the first kind of Joseph-like ideal, the ideal $\mathfrak{J}^2$ has the property that $\mathfrak{J}^2\,\cap\,\cU(\mathfrak{sp}(2n))$ is exactly equal to the Joseph ideal for $\mathfrak{sp}(2n)$. This follows immediately from Remark \ref{Cartandiff} and the quadratic relation in Theorem \ref{quadJ2}. This is intimately related to the fact that the $\mathfrak{spo}(2n|m)$-representation, of which this ideal is the annihilator ideal, is the tensor product of the minimal $\mathfrak{sp}(2n)$-representations with $\mathfrak{so}(m)$-representations, see the subsequent Theorem \ref{JosephKostant}.}
\end{remark}

Finally we can again prove that the result of \cite{Garfinkle} extends to the second notion of Joseph ideal in the current paper.
\begin{theorem}
\label{Garresult2}
Consider a two-sided ideal $\mathfrak{K}$ in $\cU(\mg)$ for $\mg=\mathfrak{spo}(2n|m)$. If $\mathfrak{K}$ has infinite codimension and the associated graded ideal $gr(\mathfrak{K})$ in $S(\mg)=\odot\mg$ satisfies
\begin{eqnarray*}
\left(gr(\mathfrak{K})\,\cap\, \odot^2\mg\right)\,\oplus\, \mg\topb\mg&=&\odot^2\mg,
\end{eqnarray*}
then $\mathfrak{K}$ is equal to the Joseph ideal $\mathfrak{J}^2$.
\end{theorem}
\begin{proof}
This result is obtained identically as Theorem \ref{Garresult}.
\end{proof}

%%%%%%%%%%%%%%%%%%%%%%%%%%%%%%%%%%%%%%%%%%%%%%%%%%%%%%%%%%%%%%%%%%%%%%%%%
\subsection{The corresponding representation of $\mathfrak{spo}(2n|m)$}
\label{secrepspo}

For $\mathfrak{spo}(2n|m)$ we consider the representations studied in \cite{Tensor}, which are classified as the only completely pointed ones, see Theorem 6 in \cite{Tensor}. This corresponds to a generalization of the metaplectic or Segal-Shale-Weil representation, which is the minimal representation of $\mathfrak{sp}(2n)$, even though the motivation to study these representations in \cite{Tensor} came from a generalization of the spinor representation of $\mathfrak{so}(m)$. 

The minimal representation of $\mathfrak{sp}(2n)$ on $n$ commuting variables is extended to a representation of $\mathfrak{spo}(2n|m)$ by adding anticommuting variables on which $\mathfrak{so}(m)$ acts. Contrary to the representation we considered for $\mathfrak{osp}(m|2n)$, this representation is unitarizable (see e.g. \cite{ChLW, Tensor}), as in the classical case. We prove that the second Joseph-like ideal is the annihilator ideal of this representation. 
\begin{definition}
\label{superGrass}
The algebra $\Lambda_{d|n}$ is freely generated by $\{\theta_1,\cdots,\theta_d,t_1,\cdots, t_n\}$ subject to the relations
\[\theta_{j}\theta_{k}=-\theta_k\theta_j\quad\mbox{for}\quad  1\le j,k\le d,\qquad t_it_l=t_lt_i\quad\mbox{for}\quad  1\le i,l\le n\]
and
\[\theta_jt_i=-t_i\theta_j\quad\mbox{for}\quad  1\le j\le d,\quad 1\le i\le n.\]
\end{definition}
This algebra is a superalgebra with unusual gradation. The commuting variables are considered as odd and the Grassmann variables are even. With this gradation the algebra is in fact a super anti-commutative algebra, $ab=-(-1)^{|a||b|}ba$ for $a,b$ two homogeneous elements of the superalgebra. Therefore this corresponds to a supersymmetric version of a Grassmann algebra.

We introduce the short-hand notation of weights $\omega_{d}=\frac{1}{2}(\epsilon_1+\epsilon_2+\cdots+\epsilon_d)$, $\omega_{d-1}=\frac{1}{2}(\epsilon_1+\epsilon_2+\cdots+\epsilon_{d-1}-\epsilon_d)$, $\nu_{n-1}=\delta_1+\delta_2+\cdots+\delta_{n-1}$ and $\nu_{n}=\delta_1+\delta_2+\cdots+\delta_{n}$.

For $\mathfrak{spo}(2n|2d+1)$, the space $\Lambda_{d|n}$ is an irreducible highest weight representation $L^{2d+1|2n}_{\omega_d-\frac{1}{2}\nu_n}\cong K^{2d+1|2n}_{\omega_d-\frac{1}{2}\nu_n}$. For $\mathfrak{spo}(2n|2d)$, the space $\Lambda_{d|n}$ decomposes into two irreducible highest weight representations $L^{2d|2n}_{\omega_d-\frac{1}{2}\nu_n}\cong K^{2d|2n}_{\omega_d-\frac{1}{2}\nu_n}$ and $L^{2d|2n}_{\omega_{d-1}-\frac{1}{2}\nu_n}\cong K^{2d|2n}_{\omega_d+\nu_{n-1}-\frac{3}{2}\nu_n}$. For both cases the weight of an element of $\Lambda_{d|n}$
\begin{eqnarray}
\label{weightmonomials}
\theta_1^{\gamma_1}\theta_2^{\gamma_2}\cdots\theta_d^{\gamma_d}t_1^{\beta_1}t_2^{\beta_2}\cdots t_n^{\beta_n}\,\,\mbox{ is }\,\,\omega_d-\frac{1}{2}\nu_n-\sum_{j=1}^d\gamma_{d-j+1}\epsilon_j-\sum_{i=1}^n\beta_{n-i+1}\delta_i.
\end{eqnarray}

The generators $\theta_i$ and $t_j$ can be combined into one notation $T_k$ for $k=1,\cdots,d+n$. The operators $\partial_{\theta_i}$ are defined by $\partial_{\theta_i}\theta_l=\delta_{il}-\theta_l\partial_{\theta_i}$, $\partial_{\theta_i}t_j=-t_j\partial_{\theta_i}$ and $\partial_{\theta_i}(1)=0$. The operators $\partial_{t_j}$ are defined by $\partial_{t_j}\theta_l=-\theta_l\partial_{t_j}$, $\partial_{t_j}t_k=\delta_{jk}+t_k\partial_{t_j}$ and $\partial_{t_j}(1)=0$. These operators $\partial_{T_k}$ generate an algebra isomorphic to $\Lambda_{d|n}$. The algebra generated by $T_k$ and $\partial_{T_k}$ is denoted by Diff$(\Lambda_{d|n})$, or Diff$(\mS_{m|2n})$.

The realization of $\mathfrak{spo}(2n|2d)$ is generated by the operators $T_iT_j$ and $\partial_{T_i}\partial_{T_j}$ with gradation inherited from the superalgebra $\Lambda_{d|n}$. The realization of\\ $\mathfrak{spo}(2n|2d+1)$ is generated by the operators $T_i$ and $\partial_{T_i}$, again with gradation inherited from the superalgebra $\Lambda_{d|n}$.

The realization inside Diff$(\Lambda_{d|n})$ of $\mathfrak{spo}(2n|m)$ for $m=2d+1$ or $m=2d$ described above and in \cite{Tensor} is denoted by $D_V$ for $V\in\mg=\mathfrak{spo}(2n|m)$, which extends to $\otimes\mg$.
\begin{theorem}
The operators $D_X$ for $X\in\otimes\mg$, where $\mg=\mathfrak{spo}(2n|m)$, introduced above satisfy the relation
\begin{eqnarray*}
D_UD_V&=&D_{U\topb V}+\frac{1}{2}D_{[U,V]}+\mu D_{\langle U,V\rangle}
\end{eqnarray*}
for $\mu=\frac{1}{4}$ and with $U,V\in\mg$.
\end{theorem}
\begin{proof}
The mapping $\mg\otimes\mg\to$Diff$(\mS_{m|2n})$ given by $U\otimes V\to D_UD_V$ is a $\mg$-module morphism, where the $\mg$ action on Diff$(\mS_{m|2n})$ is given by $V\cdot D_X=[D_V,D_X]$ for $V\in\mg$ and $X\in\otimes^2\mg$. 

Therefore if one of the subrepresentation of $\mg\otimes\mg$, which are all generated by a highest weight vector, is not in the kernel of this mapping, there has to be a highest weight vector inside Diff$(\mS_{m|2n})$ of the corresponding weight, which is constructed out of at most four generators $T_j$ or $\partial_{T_k}$. The six highest weights are given in Theorem \ref{decomposition}. For $4\delta_1$, $2\delta_1$ and $0$, such a highest weight vector is given by respectively $\partial_{t_n}^4$, $\partial_{t_n}^2$ and $1$.

For the other three highest weights, we can prove that such a highest weight vector does not exist. For $3\delta_1+\delta_2$, the only allowed vector is $\partial_{t_n}^3\partial_{t_{n-1}}$, which is not a highest weight vector since $[t_{n-1}\partial_{t_n},\partial_{t_n}^3\partial_{t_{n-1}}]=-\partial_{t_n}^4$. For $2\delta_1+2\delta_2$ the only allowed one is $\partial_{t_n}^2\partial_{t_{n-1}}^2$ which is not a highest weight vector for the same reason. The allowed vectors with weight $\delta_1+\delta_2$ are given by
\begin{eqnarray*}
v&=&\left(a_0+\sum_{j=1}^na_jt_j\partial_{t_j}+\sum_{j=1}^{n+d}a_{j+n}\theta_j\partial_{\theta_j}\right)\partial_{t_{n-1}}\partial_{t_n},
\end{eqnarray*}
for arbitrary constants $a_s$ for $s=0,\cdots,n+d$. Again, we calculate
\begin{eqnarray*}
[t_{n-1}\partial_{t_n}, v]&=&-\left(a_0+\sum_{j=1}^na_jt_j\partial_{t_j}+\sum_{j=1}^{n+d}a_{j+n}\theta_j\partial_{\theta_j}\right)\partial_{t_n}^2+a_n t_{n-1}\partial_{t_{n-1}}\partial_{t_n}^2
\end{eqnarray*}
which can only be zero if $a_s=0$ or $s=0,\cdots,n+d$.

It then follows easily that
\begin{eqnarray*}
D_UD_V&=&D_{U\topb V}+\frac{1}{2}D_{[U,V]}+\mu D_{\langle U,V\rangle}
\end{eqnarray*}
must hold for some constant $\mu$. Now if $\mu\not=\frac{1}{4}$ would hold, Theorem \ref{quadJ2} would imply that the representation $\mS_{m|2n}$ would be trivial, which is not the case.
\end{proof}
We note that the terms appearing in the right hand side are non-zero which can be checked in a straightforward calculation, but also follows from the classical case of the symplectic spinors.

\begin{theorem}
\label{JosephKostant}
The annihilator ideal in the universal enveloping algebra of the $\mathfrak{spo}(2n|m)$ representation $\mS_{m|2n}$ is equal to the Joseph-like ideal $\mathfrak{J}^2$. If $m$ is even this statement also holds for the two non-isomorphic components of the representation $\mS_{m|2n}$. In particular this implies that $ \cU(\mg)/\mathfrak{J}^2$ has infinite dimension. 
\end{theorem}
\begin{proof}
The proof is identical to that of Theorem \ref{idealrep}.
\end{proof}

%%%%%%%%%%%%%%%%%%%%%%%%%%%%%%%%%%%%%%%%%%%%%%%%%%%%%%%%%%%%%%%%%%%%%%%%%
\section{Symmetry algebra of super Laplace operator}
\label{SymmLapl}

Apart from the interest in pure algebraic results for the theory of Lie superalgebras, the main
motivation for treating the present problems is harmonic analysis on Lie supergroups. One basic 
question here aims towards understanding the kernel of the super Laplace operator $\Delta$ on 
$\mR^{m-2|2n}$ as a module over $\mathfrak{osp}(m|2n)$, see Subsection \ref{secreposp}. In the core of the problem lies the structure of the algebra of symmetry operators of $\Delta$, preserving the representation space of harmonic functions. For example, in the classical case, the commuting 
pairs of second order symmetries correspond to separation of variables for the Laplace operator and lead to 
classical coordinate systems and special function theory for orthogonal groups.   

The symmetries of the Laplace operator on $\mR^{m-2}$ are studied in \cite{MR2180410}. It follows easily that the symbol of a symmetry of the Laplace operator corresponds to a conformal Killing tensor field. The symmetries of first order generate $\mathfrak{so}(m)$. The extension of this realization of $\mathfrak{so}(m)$ to the universal enveloping algebra yields higher order symmetries. Then it can be proved that this procedure gives a higher symmetry corresponding to each conformal Killing tensor field. The kernel of this realization of $\cU(\mathfrak{so}(m))$ is exactly the Joseph ideal. This implies that there is an isomorphism between the space of symmetries of the Laplace operator on $\mR^{m-2}$, the conformal Killing tensor fields on $\mR^{m-2}$ and the quotient $\cU(\mathfrak{so}(m))/\cJ$ with $\cJ$ the Joseph ideal. In this section we discuss the corresponding statements for the super Laplace operator on $\mR^{m-2|2n}$. 

As in the classical case, a symmetry of the super Laplace operator is a differential operator $D$ which satisfies $\Delta D=\delta \Delta$ for some other differential operator $\delta$. Trivial symmetries are the ones of the form $T\Delta$ for some differential operator $T$. When we consider the algebra of symmetries, it is understood to be quotiented with respect to these trivial symmetries, so we compose equivalence classes of symmetries.

Let us consider the ambient superspace $\mR^{m|2n}$ with orthosymplectic metric $g_{ab}$ of signature
$m-1,1|2n$. The projectivization of the space ${\fam2 C}_s$ associated to the super quadric $g^{ab}x_ax_b=0$ is
the supersphere $\mS^{m-2|2n}$, regarded as a superconformal compactification of $\mR^{m-2|2n}$. 

The association $\mD_V\to\cD_V$ in equation \eqref{cDX} can be rewritten by introducing $\mC^{m|2n}$-valued functions,
\begin{displaymath}
{\Phi}_a =
\left( \begin{array}{c}
1\\
-R^2/2\\
X_i
\end{array} \right)
\,\, ,\,\,
{\Psi}^{j}_a =
\left( \begin{array}{c}
0\\
-X^j\\
\delta_{i}^j
\end{array} \right),
\end{displaymath}
which should be interpreted as $\Phi_1=1$, $\Phi_2=-R^2/2$ and $\Phi_{i+2}=X_i$ for $i=1,\cdots,m-2+2n$ and likewise for $\Psi^{j}_{a}$. It can then be checked that for $V^{ab}\in\osp$,
\begin{eqnarray}
\label{mappingfortensors}
\cD_V&=&\sum_{a,b=1}^{m+2n}\Phi_aV^{ab}\Psi^j_{b}\partial_{X^j}+V_0
\end{eqnarray}
holds for a differential operator $V_0$ of order zero, with $\cD_V$ defined in equation \eqref{cDX}. It can be proved directly that this realization of $\osp$ yields all non-trivial symmetries of degree 1 of the super Laplace operator. Composing these symmetries leads to higher order symmetries. Corollary \ref{cDcD} implies that this yields an algebra of symmetries isomorphic to the quotient of $\cU(\mg)$ with respect to an ideal containing the Joseph ideal $\mathfrak{J}^1$. Identically to Theorem \ref{idealrep} it then follows that this algebra is given by $\cU(\mg)/\mathfrak{J}^1$ if $m-2n>2$ and this would hold also for $m-2n\le 0$ in case equation \eqref{herformCartan} could be proved for $V=\mg=\osp$ for those dimensions.

In order to address the question whether this yields all higher symmetries we introduce conformal Killing tensor fields on superspace. An immediate extension of the classical definition is given underneath.

\begin{definition}
\label{defconfKill}
If $m-2-2n\not\in-2\mN$, a superconformal Killing tensor field of valence $r$ on $\mR^{m-2|2n}$ is a symmetric trace-free 
tensor field 
\[\varphi^{jk\dots l}\in C^\infty (\mR^{m-2|2n},\odot^r_0 \mC^{m-2|2n}),\] satisfying
\begin{eqnarray}\label{sckillingeq}
\partial^{(i}\varphi^{jk\dots l)_0}=0,
\end{eqnarray}
where the subscript $0$ always denotes the trace free part of a given tensor field.
\end{definition}
As was discussed in Section \ref{preli}, in \cite{OSpHarm} it was proved that the $\mathfrak{osp}(m-2|2n)$-representations $\odot^r_0 \mC^{m-2|2n}$ are irreducible and equal to $\circledcirc^r\mC^{m-2|2n}$ if $m-2n-2\not\in-2\mN$.

As in the classical case we aim to identify conformal Killing tensor fields with symbols of symmetries of the Laplace operator. The reason why traceless tensors fields are considered is because metric terms in the symbol lead to a Laplace operator, thus to a trivial symmetry. However in superspace when $m-2n-2\in-2\mN$ holds, $\odot^r \mC^{m-2|2n}$ does no longer decompose into traceless tensors and tensors containing a metric part, see the end of Section \ref{preli}. In particular, for some values of $r$, traceless tensors can contain metric terms. Therefore the quotient of $\odot^r \mC^{m-2|2n}$ with respect to tensors containing a metric term is not irreducible, but still indecomposable. As discussed at the end of Section \ref{preli}, $\odot^r_0 \mC^{m-2|2n}$ is also not isomorphic to this quotient $\odot^r \mC^{m-2|2n}/\odot^{r-2} \mC^{m-2|2n}$ for these cases, so it is important to adjust Definition \ref{defconfKill}. The reason why only the traceless part of $\partial^{(i}\varphi^{jk\dots l)}$ in Definition \ref{defconfKill} is required to be zero, is again that the metric part leads to a Laplace operator. The proper definition of superconformal Killing tensor fields therefore is given in the following definition.
\begin{definition}
\label{propdefKill}
A superconformal Killing tensor field of valence $r$ on $\mR^{m-2|2n}$ is a symmetric
tensor field 
\[\varphi^{jk\dots l}\in C^\infty (\mR^{m-2|2n},\odot^r \mC^{m-2|2n}/\odot^{r-2} \mC^{m-2|2n}),\]  where $\odot^{r-2} \mC^{m-2|2n}$ is imbedded in $\odot^{r} \mC^{m-2|2n}$ by multiplying with the metric and symmetrizing, satisfying
\begin{eqnarray}\label{sckillingeq}
\partial^{(i}\varphi^{jk\dots l)}=g^{(ij}\lambda^{k\dots l)},
\end{eqnarray}
for some tensor field $\lambda$. Denote by ${\fam2 A}_r(\mR^{m-2|2n})$ the vector space of superconformal Killing tensor fields on 
$\mR^{m-2|2n}$.
\end{definition}
If $m-2n\not\in2-2\mN$ this is identical to Definition \ref{defconfKill}.

It follows from a straightforward calculation that the symbol of a symmetry of the super Laplace operator is a superconformal Killing tensor field. As in the classical case, the question of surjectivity of the mapping extended from equation \eqref{mappingfortensors} to $\otimes\mg$ or $\cU(\mg)$ with image in the space of symmetries of the Laplace operator, can be posed in a graded way. Therefore, we only look at the highest order term of the symmetry, which is identified with its symbol. Thus we obtain a mapping $\Phi:\odot^k\mg\to C^\infty (\mR^{m-2|2n},\odot^k \mC^{m-2|2n})$
\begin{eqnarray}
\label{prekillingmap}
V^{bq\, cr\, \dots \, ds}\mapsto  \varphi^{jk\dots l}:={\Phi}_b{\Phi}_c\dots {\Phi}_dV^{bq\, cr\, \dots \, ds}
{\Psi}^j_q{\Psi}^k_r\dots {\Psi}^l_{s}.
\end{eqnarray}

Let us first consider the case $m-2n>2$. From the considerations in Lemma \ref{DUDV} it follows that if we take the quotient with respect to $\odot^{k-2} \mC^{m-2|2n}$ (considering only traceless tensors) on the right-hand side, then everything except the Cartan product $\circledcirc^k\mg$ is inside the kernel of the mapping induced by $\Phi$. This is well-defined, since the Cartan product has a complement representation in that case, see Theorem \ref{Cartanosp}. The question whether the algebra of symmetries is $\cU(\mg)/\mathfrak{J}^1$ and whether for each conformal Killing tensor field there is a symmetry with such a symbol, is therefore reduced to the question whether the mapping
\begin{eqnarray}\label{killingmap}
& & \circledcirc^r\mathfrak{g}\to {\fam2 A}_r(\mR^{m-2|2n}),
\nonumber \\   
& & V^{bq\, cr\, \dots \, ds}\mapsto  \varphi^{jk\dots l}:={\Phi}_b{\Phi}_c\dots {\Phi}_dV^{bq\, cr\, \dots \, ds}
{\Psi}^j_q{\Psi}^k_r\dots {\Psi}^l_{s}
\end{eqnarray}
is surjective.

In the classical case, the BGG
resolution \cite{lepowsky} allows to conclude that the map \eqref{killingmap} is surjective. This follows from the fact that the differential operators in Definition \ref{propdefKill} are exactly the first differential operators in the BGG resolution corresponding to the representation $\circledcirc^r\mathfrak{so}(m)$. It is 
generally believed that BGG resolutions do not exist for all finite-dimensional $\mathfrak{osp}(m|2n)$-modules. 
The reason is that finite-dimensional $\mathfrak{osp}(m|2n)$-modules correspond to rather complicated modules 
on the Lie algebra 
side according to the super duality principle, which connects parabolic categories ${\fam2 O}$ 
for the orthosymplectic Lie 
superalgebras and classical Lie algebras of BCD types, see \cite{ChLW}. On the other hand, according to the same 
super duality principle, BGG 
resolutions exist for oscillator modules of $\mathfrak{osp}(m|2n)$, see the last paragraph of Section 1.4. in \cite{ChLW}.

We leave open the surjectivity question for the representation of $\mathfrak{osp}(m|2n)$ corresponding to
super conformal Killing tensor fields. Let us remark that the problem might be geometrically 
resolved by constructing a prolongation of the overdetermined system of the superconformal 
Killing tensor differential operator, \eqref{sckillingeq}. In the case the surjectivity
hypothesis is fulfilled, the isomorphism between symmetries of the Laplace operator, conformal Killing tensor fields and the quotient $\cU(\mg)/\mathfrak{J}^1$ follows.

Finally we focus on the case $m-2n\le 0$. Even though we did not define the Joseph ideal for $m-2n\in\{1,2\}$ we can include them in this part because of the graded version of the isomorphism in question. We start again from the map \eqref{prekillingmap}, which attaches the symbol of a symmetry to each element in $\odot^k\mg$, but now assume $m-2n\le 2$. After quotienting out trivial symmetries, the question is whether the mapping
\begin{eqnarray}
\label{actualKillmap}
& & \otimes^r\mathfrak{g}/I_r\to {\fam2 A}_r(\mR^{m-2|2n}),
\nonumber \\   
& & V^{bq\, cr\, \dots \, ds}\mapsto  \varphi^{jk\dots l}:={\Phi}_b{\Phi}_c\dots {\Phi}_dV^{bq\, cr\, \dots \, ds}
{\Psi}^j_q{\Psi}^k_r\dots {\Psi}^l_{s}
\end{eqnarray}
is surjective with $I_r$ as defined in section \ref{secCartan}. But contrary to the case $m-2n>2$, it is not yet proved that this mapping \eqref{actualKillmap} is injective, injectivity would follow from equation \eqref{herformCartan} if this could be proven.

We conclude this section with some observations about the possibility that this mapping could be an isomorphism even for $m-2n\le 2$.

It is expected that $\otimes^r\mg/I_r$ will not be irreducible and not isomorphic to $\beta_r=\cU(\mg)\cdot X_{\epsilon_1+\epsilon_2}^{\otimes r}$, see Theorem \ref{Cartanosp}.
 
First let us consider the cases $m-2n\in 2-2\mN$. Then the embedding of the space $\odot^{r-2}\mC^{m-2|2n}$ into $\odot^r\mC^{m-2|2n}$ does not always have a complement representation, which leads to the fact that $\odot^r\mC^{m-2|2n}/\odot^{r-2}\mC^{m-2|2n}$ is not irreducible (but still indecomposable). This non-complete decomposability will be inherited by the conformal Killing tensor fields on the right-hand side of equation \eqref{actualKillmap}. The same thing seems to happen on the left-hand side as well, as we will point out now. This can be seen for the case $m-2n=2$ from the proof of Lemma \ref{DUDV}. Take $X\in K^{m|2n}_{2\epsilon_1}\subset\otimes^2\mg$, then $\cD_X$ has a non-zero highest order term, which is the symbol of a symmetry. However, this highest order term is proportional to the Laplace operator, so cancels out after taking the quotient with respect to trivial symmetries. From the proof of Lemma \ref{subA} we find that $m-2n=2$ is exactly the case when $K_{2\epsilon_1}^{m|2n}$ does not have a complement representation and $K_{2\epsilon_1}^{m|2n}$ is inside the representation generated by the highest weight vector in $\mg\otimes\mg$, see Theorem \ref{cases012}. So the procedure of quotienting out trivial symmetries seems to lead to non-complete reducibility of the left-hand side of equation \eqref{actualKillmap} and the right-hand side for the same values of the dimensions and degrees.

In case $m-2n\in 1-2\mN$, quotienting out the metric terms on the right-hand side does not lead to reducible representations, since Definition \ref{defconfKill} can be used. Still the left-hand side is expected to be reducible, so in order to have an isomorphism \eqref{actualKillmap} and therefore non-complete reducibility on the right-hand side, another mechanism needs to be found (compared to the case $m-2n\in 2-2\mN$). For $M=1$, the clue is again given by the calculation in the proof of Lemma \ref{DUDV}. From that proof it is clear that $X\in K_0^{m|2n}$ is already mapped to zero in the mapping in equation \eqref{prekillingmap} (since $\cD_X$ has zero highest order term), before quotienting out metric terms as in \eqref{actualKillmap}. The case $M=1$ exactly corresponds to the case where $K_0^{m|2n}$ does not have a complement representation and is a subrepresentation of the representation generated by the highest weight vector in $\mg\otimes\mg$, see the proof of Theorem \ref{cases012}. So in this case, the non-complete reducibility on the right-hand side of equation \eqref{actualKillmap} (which is needed in order to be able to have an isomorphism) now can come from the definition of symmetries of the Laplace operator, i.e. equation \eqref{sckillingeq}.

\end{document}